\documentclass[11pt,reqno]{amsart}
\usepackage[all]{xy}
\usepackage{amsfonts}
\usepackage{amssymb}
\usepackage{graphicx}
\graphicspath{{FIGURES/}}
\usepackage[latin1]{inputenc}

\usepackage[pagebackref=true]{hyperref}

\usepackage{amsmath,amssymb,amsthm, latexsym, amscd}
\usepackage{tikz-cd}

\usepackage{hyperref}

\usepackage{mathtools}

\usepackage{color}

\usepackage[switch, modulo]{lineno}

\numberwithin{equation}{section}

\newtheorem{theorem}{Theorem}[section]
\newtheorem{proposition}[theorem]{Proposition}
\newtheorem{corollary}[theorem]{Corollary}
\newtheorem{lemma}[theorem]{Lemma}

\theoremstyle{definition}
\newtheorem{definition}[theorem]{Definition}
\newtheorem{example}[theorem]{Example}
\newtheorem{remark}[theorem]{Remark}

\def\<{\langle}
\def\>{\rangle}

\def\s{\sigma}

\def\G{{\Gamma}}

\def\R{{\mathbb R}}
\def\Z{{\mathbb Z}}

\def\L{{\mathcal L}}

\def\E{{\mathcal E}}
\def\G{{\mathcal G}}

\def\J{{\mathcal J}}
\def\O{{\mathcal O}}

\def\ent{{\rm ent}}

\def\dist{\mathop{\rm dist}\nolimits}

\def\dist{\mathop{\rm dist}\nolimits}

\def\1{\mathbf 1}
\def\sys{\operatorname{sys}}

\newcommand{\cf}{{\it cf.}}
\newcommand{\ie}{{\it i.e.}}
\newcommand{\eg}{{\it e.g.}}

\newcommand{\vol}{{\rm vol}}

\long\def\forget#1\forgotten{} %

\begin{document}

%\linenumbers

\title{The Karoubi-Weibel Complexity for groups}
\author[I.~Babenko]{Ivan Babenko}
\author[T.~Moulla]{Thiziri Moulla}

\thanks{Partially supported by the ANR project Min-Max (ANR-19-CE40-0014).}

\address{Universit\'e Montpellier II, CNRS UMR 5149,
Institut Montpelli\'erain Alexander Grothendieck,
Place Eug\`ene Bataillon, B\^at. 9, CC051, 34095 Montpellier CEDEX 5, France} 

\email{ivan.babenko@umontpellier.fr}

\email{thiziri.moulla@umontpellier.fr}



\begin{abstract}
Let $G$ be a finitely presented group. A new complexity called \textit{Karoubi-Weibel complexity} or \textit{covering type}, is defined for $G$. The construction is inspired by recent work of Karoubi and Weibel \cite{KW}, 
initially applied to topological spaces. 
We introduce a similar notion in combinatorial form in order to apply it to finitely presentable groups.
Some properties of this complexity as well as a few examples of calculation/estimation for certain classes 
of finitely presentable groups are considered. Finally we give a few applications of complexity to some geometric problems, namely to the systolic area and the volume entropy of groups. 
\end{abstract}

\maketitle

\section{Introduction}
Topological complexity has been a hot subject in pure and applied topology over the last few years, although it does not yet admit a universal definition. A crude measure of the complexity of a space $X$ is the size of a finite open covering of $X$ by simple subsets \ie \ contractible subspaces. This can be reformulated as follows : what is the minimum number of simple subsets, satisfying some natural conditions on their intersections, into which $X$ decomposes ?

\vspace{2mm}
For the case of triangulated spaces we can consider all triangulations on $X$. 
In this case, it is possible to measure the complexity of $X$, by calculating, for example, the minimal number of highest dimensional simplices ($\rm{dim}(X)$-simplices) needed to triangulate $X$. This complexity measures the combinatorial volume of $X$, compare with \cite{$B^3$}. Dual complexity concerns the smallest number of vertices necessary to triangulate this space.

\vspace{2mm}
The category of Lusternik-Schnirelmann, known for almost a hundred years see \cite{LSch}, can be used to measure the complexity of the topological space $X$. This complexity is a very important tool in variational calculus and it is equal to the minimal number of open subsets, contractible in $X$ that form a cover of the space.

The complexity recently introduced by Karoubi and Weibel \cite{KW}, called covering type and denoted by $\rm{ct}$, formally resembles to that of Lusternik-Schnirelmann, but it takes into account mutual intersections of open subsets of a covering. The requirement that intersections be contractible gives a better account of the topology of the space in question. 

All topological spaces below are supposed to be path-connected and locally contractible. For a topological space $X$ we consider a finite open cover $\mathcal{U}=\{U_i\}_{i=1}^k$ of $X$. We say that ${\mathcal U}$ is a \textit{good cover} of $X$, if for all $i\in\{1,...,k\}$ the open subspaces $U_{i}$ are contractible and each of their non-empty intersections is also contractible. The size of ${\mathcal U}$ denoted $\vert {\mathcal U}\vert$ is the number of open subspaces of ${\mathcal U}$. We denote by $\rm{ss}(X)$ the smallest size of a good cover of $X$. The $\rm{ct}$-complexity (covering type) of $X$, denoted by ${\rm ct}(X)$, is the minimum of the strict covering types of spaces $Y$ homotopy equivalent to $X$, such that 
$$ 
\rm{ct}(X) = \mathop{\min}_{Y\sim X} \rm{ss}(Y). 
$$
If there are no finite coverings with the prescribed 
properties the $\rm{ct}$-complexity is equal to infinity by definition.

\vspace{2mm}
Let ${\mathcal N}({\mathcal U})$ be the nerve of the above finite cover $\{U_i\}_{i=1}^k$. The finite simplicial polyhedron ${\mathcal N}({\mathcal U})$ has $k$ vertices. In the case when ${\mathcal U}$ is a good cover of $X$, $\mathcal N$ and $X$ are homotopy equivalent. Note that $\mathcal N$ has a good cover formed by open stars of the vertices. So in the previous definition, we can minimize the number of vertices of a finite simplicial complex $Y$ where $Y$ runs over the set of complexes homotopy equivalent to $X$. 

\vspace{2mm}
For an $m$-dimensional simplicial complex, finding a lower and an upper bounds for the minimal number of $m$-simplices that make up this complex and for its covering type are quite different. Recently, Adisprasito, Avvakumov and Karasev in \cite{AAK}, gave the following subexponontial upper bound for the covering type of the real projective spaces $\R P^m$ : 
$$ 
\rm{ct}(\R P^m)\leq \exp\left\lbrace \left( \frac{1}{2}+o(1)\right) \sqrt{m+1}\log (m+1)\right\rbrace. 
$$ 
The polynomial lower bound $\rm{ct}(\R P^m)\geq \frac{1}{2}(m+1)(m+2)$ was given by Arnoux and Marin \cite[§21]{AM}, see also the recent paper of Govc, Marzantowicz and Pave\v{s}i\`c \cite{GMP}. In the same time, remark that the theorem of B\'ar\'any and Lov\'asz \cite{Ba-Lo} implies that the number of $m$-simplices needed to triangulate $\R P^m$ is at least $2^m$.

\vspace{2mm}

There exist other complexities of different natures, for example the Schwarz genus of a fiber space \cite{Shvarts58}, which is closely related to the Farber complexity \cite{Farber03}. These complexities measure the minimal number of parts into which the base space of a fiber bundle decomposes, with the condition that above each part, there is a section. These two types of complexities as well as the Lusternik-Schnirelmann category never exceed the dimension of the space plus one.

Another interesting type of complexity is the so-called Matveev complexity. This is a combinatorial invariant quite useful in $3$-dimensional topology see \cite{Matveev}, we don't discuss it here.

\vspace{0.2cm}

Considering a good cover to study the topology of a space is not new. This goes back to Jean Leray's old work \cite{leray}, and was mentioned in 1952 by Andr\'e Weil \cite{weil}.

\vspace{2mm}

Throughout this article, we focus on the complexity of Karoubi-Weibel of finitely presented groups. 

\vspace{1.5mm}

\begin{definition}
	Let $G$ be a finitely presented group. One defines the \rm{KW}-complexity be the covering type of $G$ denoted by ${\rm KW}(G)$ given by
	\begin{equation}\label{eq:ct(G)}
	{\rm KW}(G) = \mathop{\min}_{\pi_1(X) = G} {\rm ct}(X).
	\end{equation}
	In the formula, we minimize the covering type of spaces with the given fundamental group $G$.
\end{definition}	  

Each finitely presented group $G$ is a fundamental group of a finite $2$-simplicial complex, so its \rm{KW}-complexity, \textit{${\rm KW}(G)$} is always finite. The following result describes the combinatorial nature of this complexity.
	
Let $K$ be a finite simplicial complex. We denote by $s_n(K)$ the number of $n$-simplices in $K$. For a finitely presented group, we set 
$$
s_0(G) = \mathop{\min}_{\pi_1(K) = G} s_0(K) 
$$
where $K$ ranges over the set of $2$-dimensional simplicial complexes of the fundamental group $G$.

\vspace{1.5mm}

\begin{theorem}\label{th:ct(G)=s_0(G)}
Any finitely presented group $G$ satisfies ${\rm KW}(G) = s_0(G)$.
\end{theorem}
The simplest example is the group with only one element, its ${\rm KW}$-complexity is equal to $1$ and its optimal complex is composed with a single vertex.

From Theorem \ref{th:ct(G)=s_0(G)}, it is obvious that the number of isomorphism classes of groups of ${\rm KW}$-complexity bounded by $T$ is finite. For a more detailed count, let's make some remarks. Kurosh's decomposition Theorem \cite{Kur} implies that any group $G$ admits a decomposition in the form
\begin{equation}\label{eq:fact.libre}
G = G_1 \ast \mathbb{F}_n,
\end{equation}
where $\mathbb{F}_n$ is the free group of rank $n$ and $G_1$ cannot be decomposed with free factors. The rank $n$ of the free group of the equation (\ref{eq:fact.libre}) is uniquely defined and called {\it free index} of $G$. We will see below that, rather often, the ${\rm KW}$-complexity is little sensitive to the free factor of (\ref{eq:fact.libre}). 
This means that the problem of counting pairwise non isomorphic groups of bounded ${\rm KW}$-complexity should be restricted to the class of groups with free index zero. In the following, we consider only the class of groups of free index equal to $0$. Denote by $\G_{\rm KW}(T)$ the set of free index zero, pairwise non-isomorphic groups of ${\rm KW}$-complexity bounded by $T$. Theorem \ref{th:ct(G)=s_0(G)} implies the following result :

\vspace{1.5mm}

\begin{corollary}\label{coro1.2}
	For any positive real number $T$, the number of groups in $\G_{\rm KW}(T)$ satisfies $$\vert \G_{\rm KW}(T)\vert \leq 2^{3T^3\log_2\Big{(}\frac{T}{\sqrt[3]{6}}\Big{)}}.$$
\end{corollary}

\medskip
Even for simple groups the exact calculation of its covering type turns out to be a rather difficult technical problem. 
One can show that for non free groups the smallest value of ${\rm KW}(G)$ is equal to $6$. It corresponds to the cyclic group $\mathbb{Z}_2$, compare with \cite[Proposition 5.2]{KW}. Complexities of abelian groups of finite rank as well as Artin and Coxeter groups are considered in Chapter 3. We give some framing estimate of its ${\rm KW}$-complexity .
Unfortunately the exact values of ${\rm KW}$-complexities of these types of groups remain an open problem. 
 
 \vspace{1.5mm}
As yet, we only know two classes of groups whose exact values of ${\rm KW}$-complexity are known: the free groups $\mathbb{F}_n$ and the surface groups. 

Let $\mathbb{F}_n$ be a free group of rank $n$ and $k={\rm KW}(\mathbb{F}_n)$. We would like to find a relation between the rank of this group and its $\rm{KW}$-complexity. Let $K_k$ be the complete graph with $k$ vertices. Fix a vertex $x_0$ in $K_k$ and contract the star $St\{x_0\}$, we obtain a bouquet of $\frac{(k-1)(k-2)}{2}$ circles and its fundamental group is free of rank $\frac{(k-1)(k-2)}{2}\geq n$. So $k$ is the least integer satisfying $n\leq \frac{(k-1)(k-2)}{2} $ and we get
$$
{\rm KW}(\mathbb{F}_n)=\left\lceil \frac{3+\sqrt{1+8n}}{2} \right\rceil. 
$$
The notation $\left\lceil x \right\rceil$ stands for the least integer $\geq x$.

We denote by $S_{g}$ the orientable surfaces of genus $g$ and let $\pi^+(g)=\pi_{1}(S_g)$ be its fundamental group. 
We denote respectively $N_q$ the non-orientable surfaces of genus $q$ and $\pi^-(q)=\pi_{1}(N_q)$. 

Recall the definition of chromatic number of orientable and non-orientale surfaces, it is a lower bound for the number of colors that suffice to color any graph embedded in the surface. We denote by ${\rm chr}(S_g)$ and ${\rm chr}(N_q)$ the chromatic numbers of these surfaces, such that 
$$ 
{\rm chr}(S_g)=\left\lceil \frac{7+\sqrt{1+48g}}{2} \right\rceil \text{ for all } g\neq 2 \text{ and } {\rm chr}(S_2)=10 $$
and 
$$ 
{\rm chr}(N_q)=\left\lceil \frac{7+\sqrt{1+24q}}{2} \right\rceil \text{ for all } \ q\neq 2,3 \text{ and } {\rm chr}(N_2)=8 , \ {\rm chr}(N_3)=9.
$$ 

The recent paper \cite{BM} of Borghini and Minian implies that for \rm{KW}-complexity of surface groups:
$$
{\rm KW}(\pi^+(g))= {\rm chr}(S_g), \ g \neq 2 \ \ \text{and} \ \ {\rm KW}(\pi^-(q))= {\rm chr}(N_q).
$$
They proved also that there is only one exceptional case where $\rm{KW}$-complexity is different from the chromatic number,
it is the case of $\pi^+(2)$ where ${\rm KW}(\pi^+(2))= 9$. The exceptional surface group $\pi^+(2)$ is particulary interesting because its $\rm{KW}$-complexity is realized by certain $2$-simplicial complex which is not homeomorphic to a surface, see \cite{BM}. It is not the case for other surface groups.

\medskip

In order to give some geometric applications of ${\rm KW}$-complexity recall the definition of the systolic area for groups, see \cite{Gro96}. Let $G$ be a finitely presented group and $X$ be a finite $2$-simplicial complex such that $\pi_1(X) = G$. Endow $X$ wih a piecewise smooth Riemannian metric $h$. The systole, denoted by $sys(X,h)$, is the shortest length of a non-contractible closed curve in the Riemannian polyhedron $(X,h)$.
Let $\text{vol}(X, h)$ be the sum of all the $h$-areas of the $2$-simplices of $X$.
We call systolic area of the group $G$ the following quantity :
$$ 
\s(G) := \mathop{\inf}_{(X, h)}\frac{\text{vol}(X, h)}{\sys(X, h)^2},
$$
where $(X, h)$ ranges over the set of all $2$-dimensional Riemannian polyhedra $(X, h)$ such that $\pi_1(X) = G$. 

\vspace{2mm}
During the last fifteen years, the study of this invariant has been deepened, see \cite{KRS, RS08} and \cite{$B^3$}. It is known that $\s(G) = 0$ if and only if $G$ is a free group. Otherwise, a universal lower bound has been given in \cite{RS08} : $$ \s(G) \geq \frac{\pi}{16}.$$  

However, several questions remain open, for example : does every non-free finitely presented group $G$ satisfies :
$$
\begin{matrix}
1. & \s(G) & \geq &\frac{2}{\pi} \ \ \ \ \ ?\\
\hspace{3mm}
2. \ \ & \s(G \ast \Z) & = & \s(G) \ ?
\end{matrix}
$$

The constant in 1. corresponds to the systolic area of $\R P^2$, see \cite{Pu52}. So the first question is equivalent to saying that the systolic area of each simplicial complex with non-free fundamental group is at least $\s(\R P^2)$.  

\vspace{2mm}
The equality in 2. may be generalized to the conjecture that the systolic area does not depend on the free factors

\vspace{1.5mm}

\begin{theorem}\label{th:sigma(G)-ct(G)}
	Let $G$ be a finitely presented group, then : 
$$
\s(G) \leq \frac{1}{27\pi}\left({\rm KW}(G)\right) ^3. 
$$

\noindent
If in addition $G$ has zero free index then :
$$
\frac{1}{576}{\rm KW}(G) \leq \s(G). 
$$
\end{theorem}

\medskip

In order to give the definition of the volume entropy of a group, start by recalling the definition of this entropy for finite simplicial complexes, for more details see \cite{BS21}. Consider a finite simplicial complex $X$ of dimension $m$ equipped with a piecewise Riemannian metric $h$. Denote by $\hat{X}$ the universal cover of $X$ and by $\hat{h}$ the lift of $h$. Take $q \in \hat{X}$ a point and let $\hat{B}_q(R)$ be the geodesic ball of radius $R$ centered at $q$ in $X$. Put $$ \ent(X, h) := \mathop{\lim}_{R \rightarrow \infty}\frac{\log\left( \vol\left( \hat{B}_q(R)\right) \right) }{R}, $$      
where the volume means the $m$-dimensional Hausdorff mesure corresponding to the lifted metric $\hat{h}$.

It is well known that this limit exists and does not depend on the chosen point $q$. It is also known that $\ent(X, h)$ is strictly positive if and only if $\pi_1(X)$ is a group of exponential growth.

The quantity $\ent(X, h)(\vol(X, h))^{\frac{1}{m}}$ remains invariant under homotheties $h \longleftrightarrow \lambda^2h$ and we define the minimal volume entropy of $X$ as follows :

\begin{equation}\label{eq:omega(X)}
\omega(X) := \mathop{\inf}_h \ent(X, h)\left( \vol(X, h)\right) ^{\frac{1}{m}},
\end{equation}

For a non-free finitely presented group $G$, its volume entropy is then 

\begin{equation}\label{eq:omega(G)}
\omega(G) := \mathop{\inf}_{\pi_1(X) = G} \omega(X),
\end{equation}
where $X$ ranges over the set of two-dimensional complexes of fundamental group $G$.

Note that formally applied to free groups, this definition gives zero when $m=2$ in (\ref{eq:omega(X)}). The free groups are naturally one-dimensional objects and to include them in the general concept, we have to reduce ourselves in (\ref{eq:omega(G)}) to one-dimensional complexes, in other words, finite metrized (or weighted) graphs.
Minimal volume entropy for arbirary finite graphs is completely studied by Lim \cite{Lim} and by 
Kapovich and Nagnibeda \cite{KN} for the case of 
$3$-valented graphs. See also later work of McMullen \cite{McMul} for an alternative proof. The results of these articles imply directly the explicit value for the volume entropy of free groups :
$$
\omega(\mathbb{F}_n) = 3(n-1)\log 2.
$$

\vspace{1.5mm}

\begin{theorem}\label{th:omega(G)}
Each non-free finitely presented group $G$ satisfies
	$$\omega(G) \leq \frac{1}{3}\log \left({\rm KW}(G)\right) \left( {\rm KW}(G) \right)^{\frac{3}{2}}.$$
\end{theorem}

Note that in general, a universal lower bound can not exist for this previous inequality. There are groups of arbitrarily large  
${\rm KW}$-complexity and of subexponential growth for which the volume entropy is equal to zero. On the other hand, we currently know several classes of groups whose volume entropy is positive, for example the surface groups.
Another, recently found class of groups, see \cite{BC21}, consists in right-angled Artin groups such that the corresponding graph is not a forest but  has no cycles of length $3$. 
 
Finding a lower bound depending on the ${\rm KW}$-complexity for these last classes of groups remains an open question. 

\vspace{2mm}
This paper is organized as follows, in the next section, we give the proofs of the first two results presented above and a technical lemma which will be used repeatedly throughout this paper. In Section 3. we study some groups : right-angled, large and extralarge Artin/Coxeter groups, cyclic and abelian groups, and estimate the covering type of each of them. In the last section, we prove our two geometric applications \ie, Theorem \ref{th:sigma(G)-ct(G)} and Theorem \ref{th:omega(G)}.

\vspace{0.2cm}

Acknowledgements. The autors expresse their especial gratitude for Guillaume Bulteau, Daniel Massatr and Pierre de la Harpe for their multiple remarks and comments which furthered to improove the presented text.

\vspace{3.5mm}

\section{Karoubi-Weibel Complexity}

\vspace{3mm}

All the topological spaces considered in the future are supposed  path-connected and locally contractible.

\vspace{1.5mm}

\begin{proof}[Proof of Theorem \ref{th:ct(G)=s_0(G)}] 

Let $G$ be a finitely presented group and $X$ be a topological space such that $\pi_1(X) = G$.
Suppose that $\rm{ct}(X)={\rm KW}(G)$, \ie \ $X$ admits a cover $\mathcal U$ with ${\rm KW}(G)$ open contractible subsets and each non-empty intersection of these subsets is also contractible. The Nerve Lemma (see \cite{hatcher} and \cite{dug}), implies that there is a weak homotopy equivalence between $X$ and its nerve ${\mathcal N}({\mathcal U})$. Therefore the $2$-skeleton of the nerve, ${\rm Sk}^2({\mathcal N}({\mathcal U}))$, satisfies : 
$$
 \pi_1({\rm Sk}^2({\mathcal N}({\mathcal U})))=G , 
$$ 
and 
$$
\ s_0({\rm Sk}^2({\mathcal N}({\mathcal U}))) = {\rm KW}(G). 
$$
This implies that $s_0(G) \leq {\rm KW}(G)$.

To obtain the second inequality, we take a $2$-simplicial complex $K$ such that $\pi_1(K)=G$ and $s_0(K)=s_0(G)$. The open stars of vertices of $K$ form a good cover of $K$ with $s_0(G)$ open subsets.   

\end{proof}
 
\begin{proof}[Proof of Corollary \ref{coro1.2}]
	
Note by $\G_{\kappa}(T) $ the set of isomorphism classes of groups with free index zero and of simplicial complexity bounded by $T$ with $T\geq 2$. According to Theorem 1.3 of \cite{$B^3$} we have: 
\begin{equation*}
\left| \G_{\kappa}(T) \right|\leq 2^{6T\log_2T},
\end{equation*}
\vspace{2mm}
where $\left| \G_{\kappa}(T) \right|$ is the number of groups in $\G_{\kappa}(T) $. If $G\in\G_{\rm KW}(T) $ \ie, ${\rm KW}(G)\leq T$, this means that  $\kappa(G)\leq \frac{T(T-1)(T-2)}{3!}\leq \frac{T^{3}}{6}$, then 
$$
 \G_{\rm KW}(T)  \subseteq \G_{\kappa}\left( \frac{T^{3}}{6}\right) . 
$$ 
Therefore we obtain: 
$$ 
\left| \G_{\rm KW}(T) \right| \leq \left| \G_{\kappa}\left(\frac{T^{3}}{6}\right) \right| \leq 2^{T^{3}\log_2\left( \frac{T^{3}}{6}\right) } 
$$
which gives $ \left| \G_{\rm KW}(T) \right|\leq 2^{3T^3\log_2(T)} $.
\end{proof}

We will use the following technical Lemma in several estimates of the ${\rm KW}$-complexity.

\vspace{1.5mm}

\begin{definition}\label{plongement}
	Consider a simplicial complex $X$ and a simplicial subcomplex $Z \subset X$. We say that the embedding $Z \hookrightarrow X$ is maximal if for all simplex $\Delta^m\subset X$ such that its $1$-skeleton is included in $Z$ then the simplex $\Delta^m \subset Z$.	 
\end{definition}

Let $\{X,Y\}$ be a pair of simplicial complexes. Suppose that a simplicial complex $Z$ embeds like a simplicial subcomplex in each complex $X$ and $Y$  
$$ X \mathop{\hookleftarrow}^i Z\mathop{\hookrightarrow}^j Y. $$ 
Consider a pseudo-simplicial complex $W = X\underset{Z}{\cup} Y$ obtained by the gluing together $X$ and $Y$ along $Z$. The following Lemma gives us the sufficient conditions for $W$ to be a simplicial complex.

\vspace{1.5mm}

\begin{lemma}\label{lemma:lem} 
Assume that the following conditions are satisfied :
\begin{enumerate}
	
	\item For all two vertices $v_{1}$ and $v_{2}$ in $Z$, there is at most one edge $[v_1,v_2]$ in $W$ that connects them.
	\item At least one of the two embeddings $i$ or $j$ is maximal.
\end{enumerate}
Then $W$ is a simplicial complex.
\end{lemma}

\begin{proof}[Proof of Lemma \ref{lemma:lem}]
Suppose that $W$ is not a simplicial complex, this means that there exist two simplexes $\Delta_1$ and $\Delta_2$ in $X$ and $Y$ respectively, such that $\Delta_1\cap\Delta_2=M\subset Z$ where $M\neq\emptyset$ is not a simplex.

\begin{enumerate}
		\item If $\rm{Sk}^1(M)$ is not a complete graph then there exist in $M$ two vertices $v_1$ and $v_2$ which are not related by an edge in $M$, but there exist two edges $[v_1, v_2]^{(1)}\in \Delta_1 $ and $[v_1, v_2]^{(2)}\in \Delta_2$ connecting $v_1$ and $v_2$, this contradicts the condition (1).
		\item If $\rm{Sk}^1(M)$ is a complete graph with $m+1$ vertices, then there are two simplexes $\Delta_1^m\subset \Delta_1$ and $\Delta_2^m\subset \Delta_2$ such that $\rm{Sk}^1(M)=\rm{Sk}^1(\Delta_1^m)= \rm{Sk}^1(\Delta_2^m)$. If $Z$ is maximal in $X$, we get $\Delta_1^m \subset Z \subset X$ this implies that $\Delta_1^m=\Delta_2^m$ and then $M$ is a simplex. 
\end{enumerate}
Which completes the proof.
\end{proof}

\vspace{1.5mm}

\begin{remark}
       \begin{enumerate}
       \
		\item The condition \textit{$(2)$} of Lemma \ref{lemma:lem} is not necessary. Fix $\Delta^3$ a 3-dimensional simplex and let $W = \partial\Delta^3$ be the boundary of this simplex and $Z=\rm{Sk}^1(\Delta^3)$ its $1$-skeleton. We take $X$ as the union of $Z$ and two $2$-simplexes of $W$, and  also $Y$ as the union of $Z$ 
and two remaining $2$-simplexes of $W$. We can see that the embeddings of $Z$ into $X$ and $Y$ are not maximal.  
		\
		\item The condition \textit{$(1)$} of the Lemma \ref{lemma:lem} is however necessary. We illustrate this condition by the following example. Take the $2$-dimensional torus $T^2=S^1\times S^1$ and consider its minimal triangulation, see Figure \ref{triangulationTore}, which contains $7$ vertices. This triangulation is unique and its $1$-skeleton is a complete graph $K_7$. $X$ and $Y$ are both tori, the vertices of each of them are denoted by $x_i$ and $y_i$ respectively. We choose two adjacent triangles $[x_1,x_2,x_4]$ and $[x_2,x_3,x_4]$ in $X$ ($[y_1,y_2,y_4]$ and $[y_2,y_3,y_4]$ respectively in $Y$) as in Figure \ref{triangulationTore} then we remove them, so we obtain two tori with holes. Take the subcomplex $Z=\partial (X\setminus \{[x_1,x_2,x_4]\cup [x_2,x_3,x_4]\})$ in $X$ and $Z=\partial (Y\setminus [y_1,y_2,y_4]\cup [y_2,y_3,y_4])$ in $Y$. There are two ways to glue $X$ and $Y$ along $Z$ to get a surface of genus $2$. The first consists to identify the vertices $x_i$ with $y_i$ for $i\in\{1,2,3,4\}$. The second is to identify the vertices $x_i$ with $y_{i+1}$, $i\in\{1,2,3\}$ and $x_4$ with $y_1$. Each of the two gluing gives a semi-triangulation of the surface $W$. The first gluing is not a triangulation because there are in $W$ two edges (which are dark in Figure \ref{triangulationTore}) such that their intersection is two distinct vertices $x_1=y_1$ and $x_3=y_3$. But in the second case, the condition \textit{$(1)$} of the Lemma \ref{lemma:lem} is satisfied and we get a triangulation of the surface $W$. This triangulation contains $10$ vertices and it provides a minimal triangulation of the surface of genus $2$.

%\bigskip

%%%%%%%%%%%%%%%%%%%%%%%%%%%%%%%%%%%%%%%%%%%%%%%%%%%%%%%%%%%%%%%%%%%%		
	         	\begin{figure}[h]
				\includegraphics[scale=1.1]{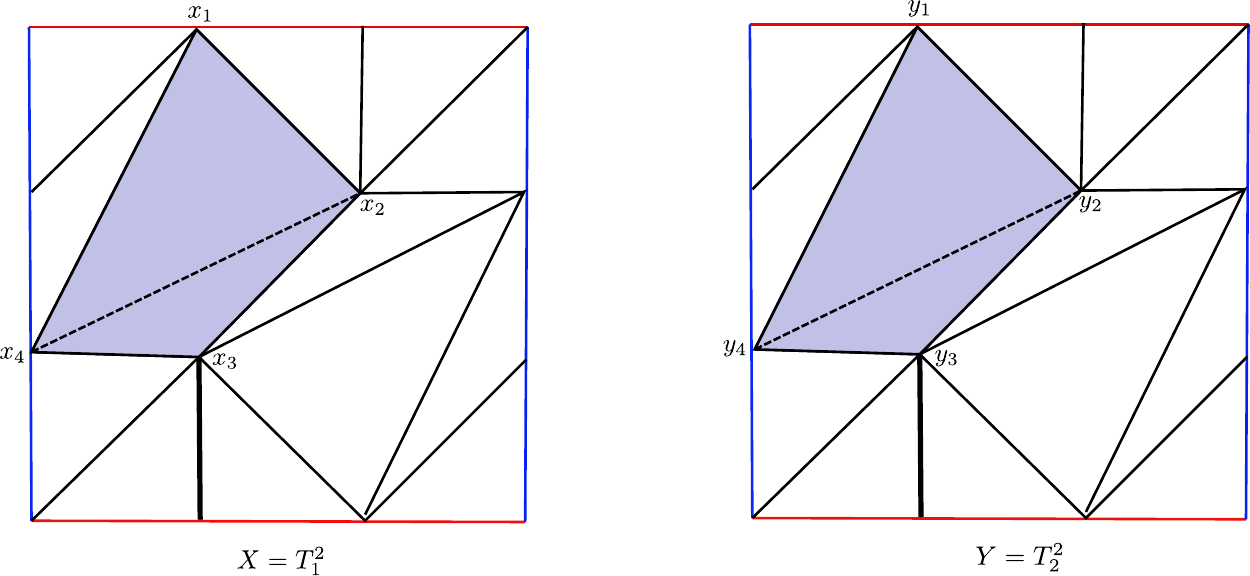}
				\caption{\label{triangulationTore}Minimal triangulation of a torus $T^2$}
				\end{figure}		
	\end{enumerate} 
 
%%%%%%%%%%%%%%%%%%%%%%%%%%%%%%%%%%%%%%%%%%%%%%%%%%%%%%%%%%%%%%%%%%%%
\end{remark}

\vspace{1.5mm}

\begin{proposition}\label{prop:proprietes}
Let $G_i$,  $i=\{1,2\}$ be two finitely generated groups. Then
\begin{enumerate}
\item ${\rm KW}(G_1 \times G_2) \leq {\rm KW}(G_1) \times {\rm KW}(G_2)$,

\item ${\rm KW}(G_1 \ast G_2) \leq {\rm KW}(G_1) + {\rm KW}(G_2) - 3 + a$,
where $a = 0$ if $G_1$ and $G_2$ are both non-free and $a=1$ otherwise.

\item If $H<G$ is a subgroup of index $k$ then $ {\rm KW}(H)\leq k \ {\rm KW}(G). $

\item If ${\rm KW}(G)=n$ then the Betti numbers of $G$ satisfy $b_k(G) \leq C_{k+1}^{n-1}$, $k=1, 2$.
\end{enumerate}
\end{proposition}

\begin{proof}[Proof of Proposition \ref{prop:proprietes}]

The first inequality is analogous to the $\rm{ct}$-complexity of a product of simplicial complexes given in \cite[Remark 7.4]{KW}.

\vspace{2mm}

To prove the second inequality \textit{$(2)$}, let us consider two simplicial complexes $X_1$ and $X_2$ satisfying $\pi_1(X_i)=G_i$ and ${\rm KW}(G_i)={\rm ct}(X_i)$ for $i\in\{1,2\}$. We distinguish three cases :
    
\begin{enumerate}
   	
\item[$\bullet$] If $G_1$ and $G_2$ are both non-free groups then $G_1\ast G_2$ is a fundamental group of the simplicial complex $X_1\underset{\Delta^2}{\cup}X_2$ obtained by the gluing of $X_1$ and $X_2$ along a $2$-simplex $\Delta^2\subset X_i$, $i=\{1,2\}$ so $$ {\rm KW}(G_1\ast G_2)\leq {\rm KW} (G_1)+{\rm KW} (G_2)-3, $$

\item[$\bullet$] If $G_1$ or $G_2$ is free group or both of them then $G_1\ast G_2$ is a fundamental group of the simplicial complex $X_1\underset{Z}{\cup}X_2$ obtained by the gluing of $X_1$ and $X_2$ along an edge $Z\subset X_i$, $i=\{1,2\}$ so $$ {\rm KW}(G_1\ast G_2)\leq {\rm KW} (G_1)+{\rm KW} (G_2)-2. $$ 

\end{enumerate}

The third inequality is immediate.

If $X$ is a $2$-complex such that $\pi_1(X) = G$ one has $b_1(X)=b_1(G)$ and $b_2(X)\geq b_2(G)$. So the last property is direct from Theorem $3.3$ of \cite{KW}. 

\end{proof}

\vspace{1.5mm}    
     
\begin{remark}
It seems that inequality \textit{$(1)$} of Proposition \ref{prop:proprietes} is never exact if the two groups are not trivial. 
A simple example is $G_1=G_2=\Z$. In this case ${\rm KW}(G_1) = {\rm KW}(G_2) = 3$ but ${\rm KW}(G_1 \times G_2) = 7$. 
On the other hand the inequality \textit{$(2)$} is optimal as shows the example of $G_1 = \mathbb{F}_n$ where $n=\frac{(k-1)(k-2)}{2}$ with $k$ an integer, and $G_2= \Z $. Here ${\rm KW}(G_1)=k$, ${\rm KW}(G_2)=3$ and ${\rm KW}(G_1\ast G_2) = {\rm KW}(\mathbb{F}_{n+1})=k+1$. 
 
However, some examples show that the behavior of ${\rm KW}(G_1 \ast G_2)$ can be quite irregular according to the complexity of the factors.

The third inequality seems to be far from optimal. A good constant in (3) remains unknown. 

\end{remark}

\vspace{3mm}

\section{The ${\rm KW}$-complexity of certain groups}
\vspace{3mm}

In this part we study the ${\rm KW}$-complexity of certain finitely presented groups, namely Artin and Coxeter groups, Abelian, in particular cyclic groups. For Artin and Coxeter groups, we follow the same notations as in \cite{BrieskSai72}.

\vspace{2mm}

Let $ E= \{a_i\}_{i=1}^n$ be a finite set of elements and $F_E$ a free group generated by this set. Denote $\overline{\mathbb{N}} = \mathbb{N} \cup \{\infty\}$ and taking $a_i$, $a_j$ two distinct elements of $E$. For all $m_{ij}\in\overline{\mathbb{N}}$, $<a_ia_j>^{m_{ij}}=a_ia_ja_i \dots$ denotes an alternating product of $a_i$ and $a_j$ of length $m_{ij}$, such that   

\begin{enumerate}
	\item[$\bullet$] if $m_{ij}=2k_{ij}$, $k_{ij}\in\mathbb{N}$ so $<a_ia_j>^{m_{ij}}=(a_ia_j)^{k_{ij}}$,
	\item[$\bullet$] and if $m_{ij}=2{k_{ij}}+1$ so $<a_ia_j>^{m_{ij}}=(a_ia_j)^{k_{ij}}a_i$,

\end{enumerate}  
 
Let us note by $M = (m_{ij}) $ a symmetric $(n\times n)$-matrix called Coxeter Matrix of elements $m_{ji}=m_{ij} \in\overline{\mathbb{N}}$ such that $m_{ii} = 1$ and if $i\neq j$ $m_{ij} \geq 2$.

\vspace{1.5mm}

\begin{definition}{\textbf{Artin Groups}} \\
An Artin group generated by the set $E$ is a group with a presentation of the following form: 
\begin{equation}\label{eq:def-Artin}
G_{A}(M)=\langle a_{1},...,a_{n}| <a_{i}a_{j}>^{m_{ij}}=<a_{j}a_{i}>^{m_{ji}}, \ i\neq j \text{ and } i,j\in\{1,...,n\} \rangle, 
\end{equation} 
Here  there is no relation between $a_i$ and $a_j$ if $m_{ij}=\infty$.
According to the integers $m_{ij}, \hskip3pt i\neq j$, there are three common types of Artin groups

\begin{enumerate}
\item Right-angled Artin groups : all elements of $M$, $m_{ij}$ are equal to $2$ or $\infty$,
\item Artin group of \textit{large type} when $m_{ij} \geq 3$,% for all $i\neq j$,
\item Artin group of \textit{extra-large type} when $m_{ij} \geq 4$.% for all $i\neq j$.
\end{enumerate}
\end{definition}

\vspace{1.5mm}

\begin{definition}{\textbf{Coxeter Groups}} \\
We call Coxeter group generated by the set $E$, a group with a presentation of the form :
 $$ G_{C}(M)=\langle a_{1},...,a_{n}|a_{i}^{2}=e, <a_{j}a_{i}>^{m_{ji}}=<a_{j}a_{i}>^{m_{ji}}, \ i\neq j \text{ and } i,j\in\{1,...,n\} \rangle. $$
 As above,
 \begin{enumerate}
 	\item If $m_{ij}=\{2,\infty\}$, $G_C(M)$ are called a right-angled Coxeter groups,
 	\item If $m_{ij}\geq 3$ for all $i\neq j, \ i,j\in\{1,...,n\}$ then $G_{C}(M)$ is a Coxeter group of \textit{large type},
 	\item If $m_{ij}\geq 4$ for all $i\neq j, \ i,j\in\{1,...,n\}$ then $G_{C}(M)$ is a Coxeter group of \textit{extra-large type}.
 \end{enumerate} 
\end{definition}

The chapter is devoted to systematical study the ${\rm KW}$-complexity of this two types of groups.

\vspace{1.5mm}

\begin{example}
The simplest example of a Coxeter group is the group of one generator $\mathbb{Z}_{2}$. Its classifying space is the infinite real projective space $\mathbb{R}P^{\infty}$ and its $2$-skeleton is the real projective plane $\mathbb{R}P^{2}$. The minimal triangulation of $\mathbb{R}P^{2}$, given in Figure \ref{$RP2$}, contains $6$ vertices. This number corresponds to the ${\rm KW}$-complexity of the group $\mathbb{Z}_{2}$ \ie \ ${\rm KW}(\mathbb{Z}_{2})=6$.

%%%%%%%%%%%%%%%%%%%%%%%%%%%%%%%%%%%%%%%%%%%%%%%%%%%%%%%%%%%%%%%%%%%%%%%%%%%%%%%%%%%%%%%%%%%%%%%%%%%%%%%%%%%%%%
 	\begin{figure}[!ht]\label{planProjectif}
 	\centering
 	\includegraphics[scale=0.325]{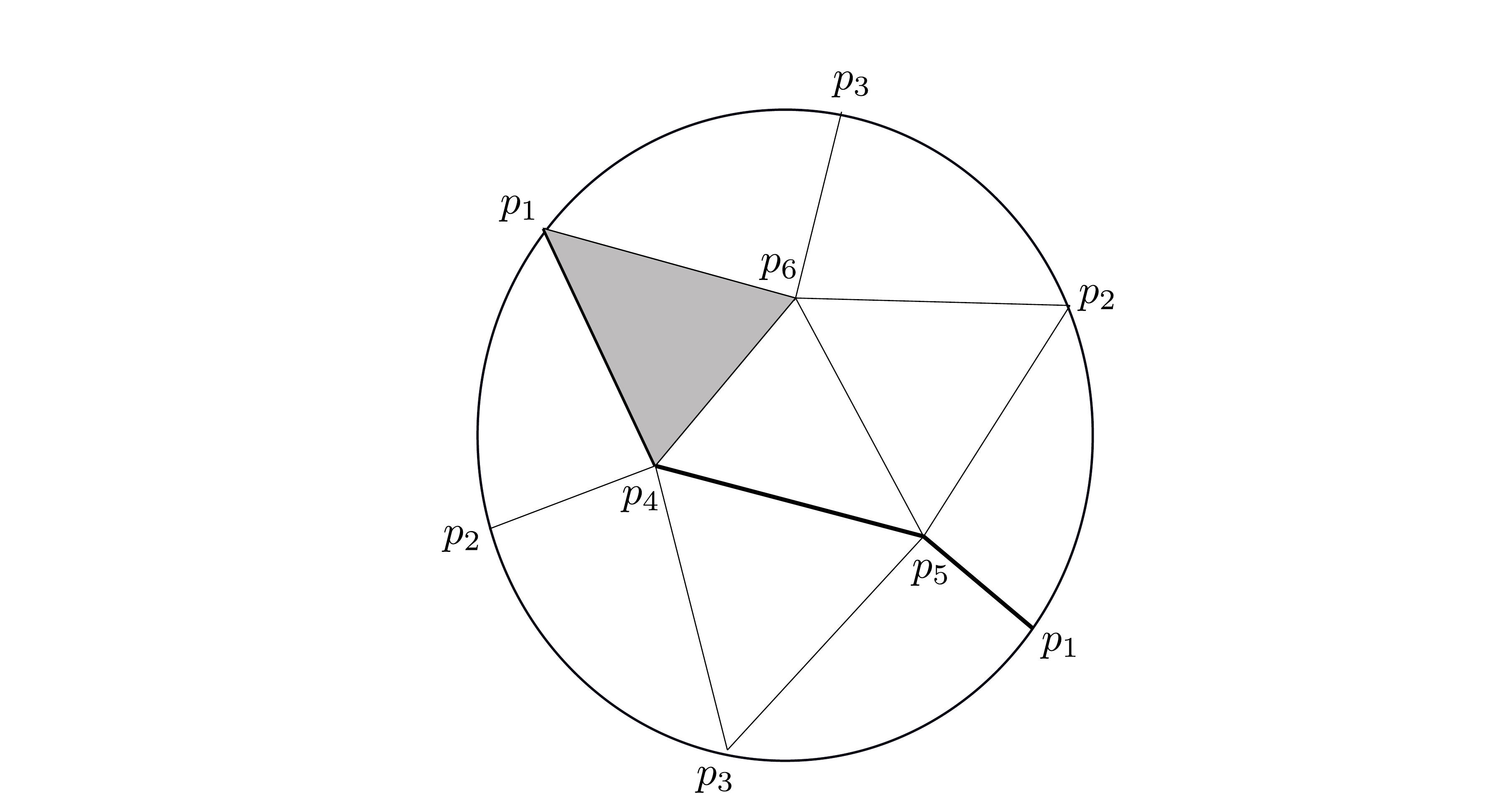}
 	\caption{\label{$RP2$} Minimal Triangulation of $\mathbb{R}P^{2}$}
 	
 	\end{figure}
\end{example}

%%%%%%%%%%%%%%%%%%%%%%%%%%%%%%%%%%%%%%%%%%%%%%%%%%%%%%%%%%%%%%%%%%%%

Let $M$ contains $2m$ finite entries and let us consider a right-angled Artin group $G_A(M)$, so $G_A(M)$ has $n$ generators and $m$ relations. The following result gives an estimate of the ${\rm KW}$-complexity of the right-angled Artin groups.

\vspace{1.5mm}

\begin{theorem} \label{theoreme}
	Let $G_{A}(M)$ be a right-angled Artin group with $n$ generators and $m$ relations. Then :
$$ 
k_A(n,m)\leq {\rm KW}(G_{A}(M))\leq 2(n+m)+1,
$$	
where $k_A(n,m)$ is the function of two variables, defined as follows :
$$
k_A(n,m) = \Bigg{\{} \begin{matrix}
\frac{\sqrt{8n+1}+3}{2} & \text{if} & m\leq \frac{n}{6}(\sqrt{8n+1}-3), \\
\sqrt[3]{6m} + 2   & \text{otherwise}. & 
\end{matrix}
$$ 
\end{theorem}

\begin{proof}[Proof of Theorem \ref{theoreme}] Let us start with the upper bound. For right-angled Artin group, the relations in (\ref{eq:def-Artin}) are the commutativity relations, so the $2$-complex $K(n, m)$ obtained by the presentation of $G_{A}(M)$ is naturally a subcomplex of the $2$-skeleton of a $n$-dimensional torus, $T^n$. 
	
Consider the canonical cellular decomposition of torus $T^n$ that given by the product of $n$ circles. Then   

\begin{equation}\label{eq:K(n, m)}
K(n, m) = \bigg{(}\mathop{\bigvee}\limits_{i=1}^n S^1_i\bigg{)}
\bigcup \bigg{(}\mathop{\bigcup}\limits_{\substack{i<j\\ m_{ij}<\infty}}T^2_{\{ij\}}\bigg{)} \subset \rm{Sk}^2(T^n),
\end{equation}
where the circles $S^1_i$ of the wedge sum correspond to the generators $a_i$ of $G_{A}(M)$. Each $2$-dimensional torus $T^2_{\{ij\}}$ corresponds to the commutativity relation $a_ia_j=a_ja_i$ of (\ref{eq:def-Artin}). We can see that this complex coincides with the $2$-skeleton of the classifying space of $G_{A}(M)$, see for example \cite{Char07}.

Now we give the triangulation of $K(n, m)$ as follows, on each circle of (\ref{eq:K(n, m)}) we consider the triangulation of the complete graph $K_3$ such that one vertex of this graph coincides with the base vertex of the wedge sum. For each $2$-cell $T^2_{\{ij\}}$, we take a minimal triangulation of the torus $T^2$ given in Figure \ref{triangulationTore}. The three vertical edges of the square coincide with those of the circle $S^1_i$, and the horizontal edges correspond to those of the circle $S^1_j$.     

To prove that the pseudo-triangulation given above is a triangulation, we proceed by induction.

If $m=0$ so $K(n, 0)=\bigg{(}\mathop{\bigvee}\limits_{i=1}^n S^1_i\bigg{)}$ and the conclusion is obvious. Assume that by adding $m$ tori one obtains the complex (\ref{eq:K(n, m)}) whose triangulation is supposed to be convenient. 
We now add another torus $T^2_{\{kl\}}$:
\begin{equation}\label{eq:K(n, m+1)}
K(n, m+1) = K(n, m) \bigcup_Z T^2_{\{kl\}},
\end{equation}
where $Z$ is the $1$-dimensional subcomplex belonging to both $K(n, m)$ and $T^2_{\{kl\}} $. 
	
\vspace{1.5mm}
	
To verify that $K(n,m+1)$ is a simplicial complex, we check the conditions of Lemma \ref{lemma:lem}. From the geometry of cellular decomposition of $K(n, m)$ we have three possibilities for $Z$. The first one is $Z = K_3 \vee K_3$. 
Thus this subcomplex corresponds to the wedge sum $S^1_k\vee S^1_l \subset T^2_{\{kl\}}$. 
The circles are not contractible so the embedding of $Z \hookrightarrow T^2_{\{kl\}}$ is maximal. To satisfy the condition \textit{$(1)$} of the same Lemma, note that from the triangulation of $T^2_{\{kl\}}$, two vertices of $Z$ are connected by one edge and which does not belong to $Z$ if these vertices are different from the base point $p$ and each of them belongs to the circles $S^1_k$ and $S^1_l$. In this case, these vertices do not belong to the same torus $T^2_{\{ij\}} \subset K(n, m)$ then they are not connected by an edge in $K(n,m)$. 
The second and third possibilities are respectively $Z=K_3$ and $Z= \{pt\}$. In these two cases Lemma \ref{lemma:lem}
applies obviously. 

\vspace{1.5mm}

To complete the first part of the proof, we calculate the number of vertices obtained from the triangulation of $K(n,m)$. The wedge sum in (\ref{eq:K(n, m)}) gives $2n+1$ vertices. Each torus $T^2_{\{ij\}}$ of (\ref{eq:K(n, m)}) contains in addition $2$ vertices which belong to $ T^2_{\{ij\}}\big\backslash\left\lbrace S^2_i\vee S^2_j\right\rbrace $. Then the triangulation of $K(n,m)$ contains $2n+2m+1$ vertices.

\vspace{3.5mm}

For the lower bound, let us take $k = {\rm KW}(G_A(M))$ so the number $n$ of generators and the number $m$ of relations of $G_A(M)$ satisfy : 

\begin{equation}\label{eq:n-m.majoration}
n \leq \frac{(k-1)(k-2)}{2} , 
\end{equation}
and 
\begin{equation}\label{eq:n-m.majoration2}
 m \leq \frac{(k-1)(k-2)(k-3)}{6}.
\end{equation}
Consider the positive part of Euclidean plan $(x,y)\in\R_{+}^2$ and define a parametric curve 
$$\left( x = \frac{(t-1)(t-2)}{2} \ ; \ y = \frac{(t-1)(t-2)(t-3)}{6}\right) , \hskip5pt t\geq 3,$$ whose cartesian equation is $y = \frac{x}{6}\left( \sqrt{8x+1}-3\right) $. This curve divides the first quadrant of the plane into two parts. The lower part $\left\lbrace m \leq \frac{n}{6}\left( \sqrt{8n+1}-3\right) \right\rbrace $ and the upper part $\left\lbrace m \geq \frac{n}{6}\left( \sqrt{8n+1}-3\right) \right\rbrace $ 

If $(n,m)$ is in the lower part then the right-hand side of (\ref{eq:n-m.majoration}) is higher than the one in (\ref{eq:n-m.majoration2}), hence $k \geq \frac{\sqrt{8n+1}+3}{2}$. In the opposite case, \ie, $(n,m)$ is in the upper part of the quadrant, the inequality (\ref{eq:n-m.majoration}) becomes more powerful than inequality (\ref{eq:n-m.majoration2}). This implies that $k \geq t_*(m)$ where $ t_*(m) $ is the real root of the equation $(t-1)(t-2)(t-3) =6m$. Therefore $k \geq \sqrt[3]{6m} + 2$, hence the result. 

\end{proof}

For Artin groups, the number of relations is bounded by $\frac{n(n-1)}{2}$ where $n$ is the number of generators. When $m=\frac{n(n-1)}{2}$, the right-angled Artin group corresponds to  the free abelian group of rank $n$. From Theorem \ref{theoreme} we obtain :

\vspace{1.5mm}

\begin{corollary}\label{gpAbelien}
The ${\rm KW}$-complexity of free abelian group $\mathcal{A}_n$ of rank $n$ satisfies :
$$ 
\lceil(3n(n-1))^{\frac{1}{3}}\rceil+2 \leq {\rm KW}(\mathcal{A}_n)\leq n^{2}+n+1. $$	
\end{corollary}

\vspace{1.5mm}

\begin{example}
	For the abelian group $\mathcal{A}_2$ of rank $2$, the Corollary \ref{gpAbelien} implies 
$$
4 \leq {\rm KW}(\mathcal{A}_2)\leq 7,
$$
On the other hand	
this group is a fundamental group of torus $T^2$ and we know from \cite{BM} that the exact value of its ${\rm KW}$-complexity is equal to $7$.
	
	If the rank is equal to $3$, from the same corollary \ref{gpAbelien} we get $5 \leq {\rm KW}(\mathcal{A}_3)\leq 13$. The exact value of its ${\rm KW}$-complexity is not yet known.
\end{example}

\vspace{1.5mm}

\begin{theorem}\label{Coxeter}
Let $G_{C}(M)$ be a right-angled Coxeter group with $n$ generators and $m$ commutation relations. Then 
 
$$ k_C(n,m)\leq {\rm KW}(G_{C}(M))\leq 5n+2m+1,$$	
where $k_C(n,m)$ is a function of two variables defined in the following form :
	$$ 
	k_C(n,m) = \Bigg{\{} \begin{matrix}
	\frac{\sqrt{8n+1}+3}{2} & \text{ if } & m\leq \frac{n}{6}\left( \sqrt{8n+1}-3\right) , \\
	\sqrt[3]{6(m+n)} + 2   & \text{ otherwise }. & 
	\end{matrix}
	$$ 
\end{theorem}

\begin{proof}[Proof of Theorem \ref{Coxeter}]
We apply the same reasoning as above. The relations $a_{i}^{2}=e$, $ i\in \{1,...,n\}$ of $G_{C}(M)$ correspond in the classifying space of the group to the infinite real projective spaces $\R P^\infty $. We define the $2$-complex $P(n,m)$ such that $\pi_{1}(P(n,m))=G_{C}(M)$ by its cellular decomposition such that : 
 
\begin{equation}\label{eq:P(n, m)}
P(n, m) = \bigg{(}\mathop{\bigvee}\limits_{i=1}^n \R P^2_i\bigg{)}
\bigcup \bigg{(}\mathop{\bigcup}\limits_{\substack{i<j\\ m_{ij}<\infty}}T^2_{\{ij\}}\bigg{)} \subset 
{\rm Sk}^2\Big{(}\mathop{\prod}_{i=1}^n\R P^{\infty}_i\Big{)}.
\end{equation}
We can see that this decomposition is the corresponding decomposition of the $2$-skeleton of the classifying space of $G_{C}(M)$. 	
\vspace{1.5mm}
To triangulate properly this cellular complex (\ref{eq:P(n, m)}), we shall use the minimal triangulation of the torus $T^2$ and of the real projective plane $\R P^2$, see respectively Figures \ref{triangulationTore} and \ref{$RP2$}. For the real projective planes the vertex $p_1(i)$ is associated to the base vertex of wedge sum (\ref{eq:P(n, m)}). The generators $a_i$ for all $i\in\{1,...,n\}$ correspond to the concatenation of three edges $[p_1(i), p_4(i)], [p_4(i), p_5(i)], [p_5(i), p_1(i)]$ of this triangulation.  
 
Assume that the $2$-cells of $P(n,m)$ are triangulated. By applying Lemma \ref{lemma:lem} in the same way as in the proof of Theorem \ref{theoreme}). We see that the triangulations of all $2$-cells of $P(n,m)$ are coherent and so we obtain a global triangulation of $P(n,m)$. Now we calculate the number of vertices derived from this triangulation. Then 
$$
 s_0\left( {\rm Sk}^{2}\left( BG_{C}(M)\right) \right) \leq 2(n+m)+1+3n=5n+2m+1, 
 $$
where $\rm {BG_{C}(M)}$ is the classifying space of $G_C(M)$. This implies $$ {\rm KW}(G_{C}(M))\leq 5n+2m+1.$$

For the lower bound, we proceed as in the previous proof of theorem \ref{theoreme}, the difference is that in this case the number of relations is $r=m+n$, and this ends the proof. 
 
\end{proof}

If all generators of $G_{C}(M)$ commute two-by-two so $m=\frac{n(n+1)}{2}$ and the corresponding group is the direct sum of $m$ cyclic groups $\mathbb{Z}_2$,
\begin{equation*}
 G_{C}(M)=\underset{i=1}{\overset{n}{\oplus}}( \mathbb{Z}_2)_i. 
\end{equation*}
Then we have,  

\vspace{1.5mm}

\begin{corollary} 
	
	$\frac{4}{3}n^{\frac{2}{3}}+1\leq {\rm KW}\left( \underset{i=1}{\overset{n}{\oplus}}\left( \mathbb{Z}_2 \right)_i\right) \leq n^2+4n+1.$
	
\end{corollary}

\subsection{Artin/Coxeter groups of large and extra-large type} 

\
\vspace{2mm}

Throughout this section, we focus on the study of Artin/Coxeter groups of large type. The construction used below is universal and it can be applied to all types of Artin/Coxeter groups. It becomes really interesting if groups are of the uniformly extra-large type. The last means that $m_{ij}\gg 1$ for all $i\neq j$.

We begin by the following Proposition,

\vspace{1.5mm}

\begin{proposition}\label{proposition}
Let $G$ be a finitely presented group with $n$ generators and one relation  $$G=\langle a_1, a_2, ..., a_n | w^m=v^m  \rangle $$ where $w=a_{t_1}^{\varepsilon_1}a_{t_2}^{\varepsilon_2}...a_{t_l}^{\varepsilon_l}$ and $v=a_{t'_1}^{\varepsilon'_1}a_{t'_2}^{\varepsilon'_2}...a_{t'_{l'}}^{\varepsilon'_{l'}}$ are two cyclically reduced words such that $\varepsilon_i, \varepsilon'_j\in\{-1,+1\}$, $t_i, t'_j\in \{1,...,n\}$ for all $i\in\{1,...,l\}$ and $j\in\{1,...,l'\}$, where $l$ and $l'$ are the lengths of $w$ and $v$ respectively. Then there exists a $2$-simplicial complex $K$, such that $\pi_{1}(K)=G$ and $$ s_0(K)\leq 8\log_2m+2n+\frac{3}{2}(l+l')+5. $$	
\end{proposition}

\begin{proof}[Proof of Proposition \ref{proposition}]
To construct such a complex, we apply the same telescopic process as that of the proof of Lemma 4.1 in \cite{$B^3$}.  

We start by defining two simplicial complexes $K_v$ and $K_w$. Let $S=\overset{n}{\underset{i=1}{\bigvee}}S^1_i$ be a wedge sum of $n$ circles. Each circle $S^1_i$ is associated to a generator $a_i$, $i\in\{1,...,n\}$. We triangulate $S$, by taking each circle as a concatenation of three edges and denote by $p$ the base vertex.

\vspace{1.5mm}
Take $P_l$ et $P_{l'}$ two triangulated disks as follows, the  boundaries of $P_l$ and $P_{l'}$ are triangulated into $3l$ and $3{l'}$ edges, respectively. The inside of each of the two disks is decomposed into sections, each of them containing $4$ or $5$ $2$-simplexes placed in an alternate manner (see Figures \ref{exemple} and \ref{exemple2}). The central section of $P_l$ contains $l$ or $l+1$ $2$-simplexes, and that of $P_{l'}$ contains $l'$ or $l'+1$ $2$-simplexes depending on the parity of $l$ and $l'$, respectively. In total, we get $l+1$ sections for $P_l$ and $l'+1$ for $P_{l'}$.

%%%%%%%%%%%%%%%%%%%%%%%%%%%%%%%%%%%%%%%%%%%%%%%%%%%%%%%%%%%%%%%%%%%%%%%%%%%%%%%%%%%%%%%%%%%%%%%%%%%%%%%%%%%%%%%%%%%%%%%

\begin{figure}[!ht] 
	\includegraphics[scale=0.335]{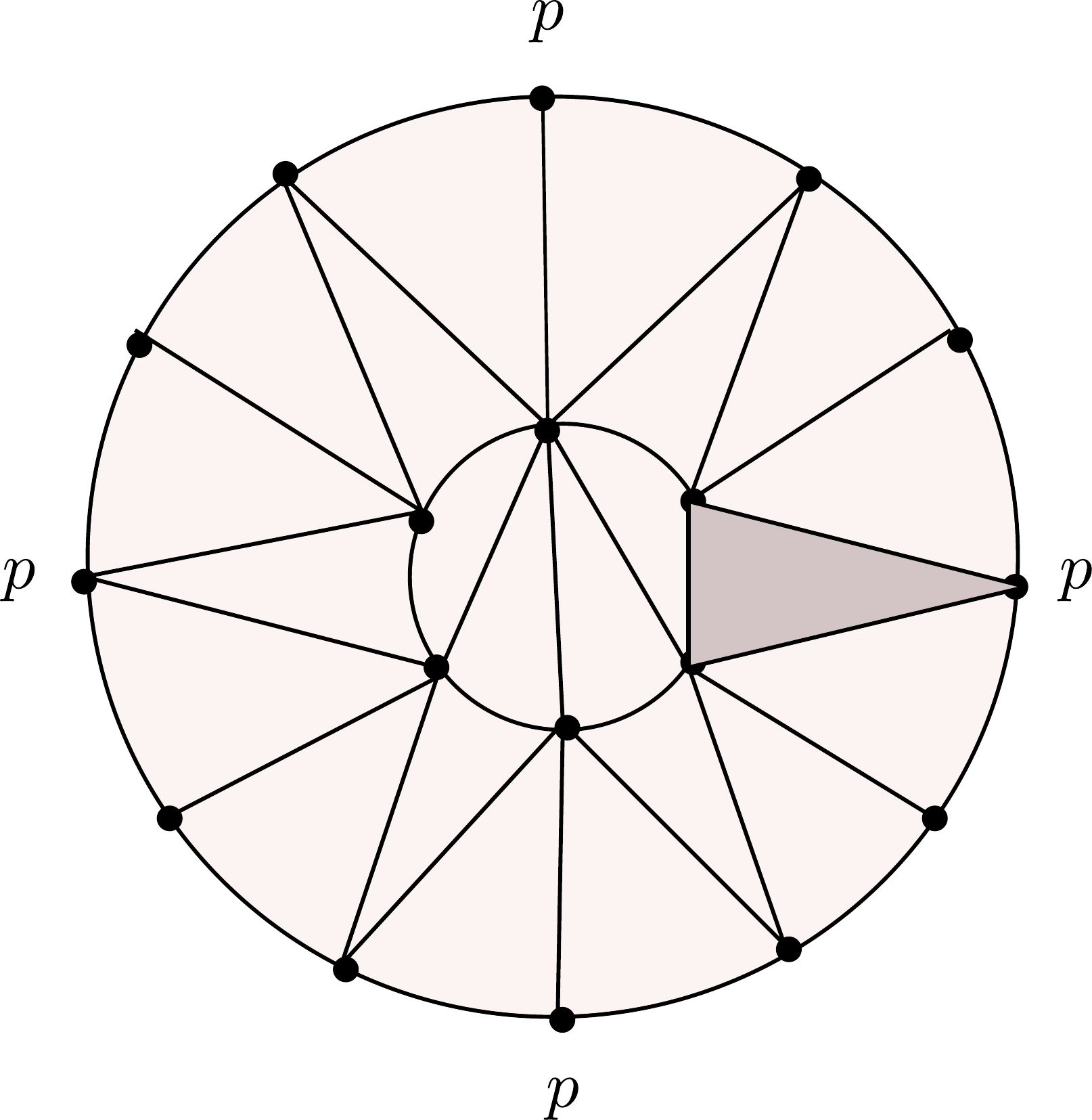}
	\caption{\label{exemple}}Triangulation of the disk $P_l$ for even $l$ ($l=4$).
\end{figure}

%%%%%%%%%%%%%%%%%%%%%%%%%%%%%%%%%%%%%%%%%%%%%%%%%%%%%%%%%%%%%%%%%%%%%%%%%%%%%%%%%%%%%%%%%%%%%%%%%%%%%%%%%%%%%%%%%%%%%%%%%%%%%%

\begin{figure}[!ht]
	\includegraphics[scale=0.325]{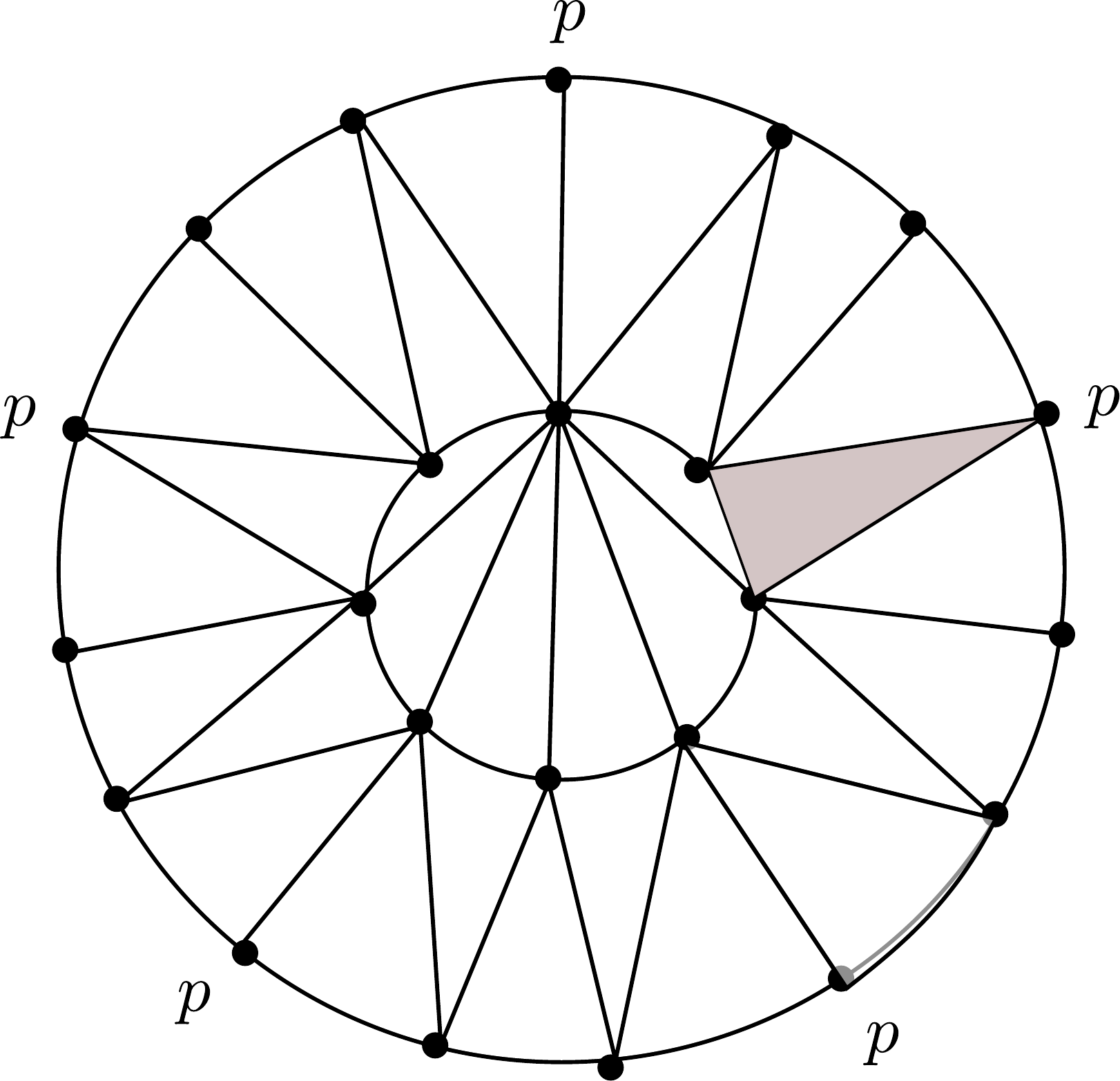}
	\caption{\label{exemple2}}Triangulation of the disk $P_l$ for odd $l$ ($l=5$).
\end{figure}

%%%%%%%%%%%%%%%%%%%%%%%%%%%%%%%%%%%%%%%%%%%%%%%%%%%%%%%%%%%%%%%%%%%%%%%%%%%%%%%%%%%%%%%%%%%%%%%%%%%%%%%%%%%%%%%%%%%%%%%%%%%%%

\vspace{2mm}
The complexes $P_w$ and $P_v$ are obtained by the gluing of $P_l$ and $P_{l'}$ with $S$ along the words $w$ and $v$, respectively. It is easy to verify that these complexes provide two triangulations.

\vspace{2mm}
Fix two $2$-simplices $\Delta_p^2$ and $\Delta^{'2}_p$ of $P_w$ and $P_v$, respectively. Each of them contains the vertex $p$, \eg, the dark triangles of Figures \ref{exemple} and \ref{exemple2}, and remove them. We get two new complexes $P'_w$ and $P'_v$ whose fundamental groups are free groups of rank $n$, and homotopy classes of curves
$\partial\Delta_p^2$ and $\partial\Delta_p^{'2}$ correspond to the words $w$ and $v$ respectively.  

\vspace{2mm}
Now, starting with $\partial\Delta_p^2$ and $\partial\Delta^{'2}_p$, we construct two M\"obius telescopes $\mathcal{J}_k$ and $\mathcal{L}_k$ of length $k$. We are going to use the  construction of the proof of Lemma 4.1 in \cite{$B^3$}.

\vspace{2mm}
Let $\{M_i\}_{i\in\mathbb{N}}$ and $ \{M'_i\}_{i\in\mathbb{N}} $ be two sequences of M\"obius bands, satisfying :

\begin{enumerate}
	
	\item $p_{i}$ is a point in $\partial M_{i}$, and $p'_{i}$ is a point in $\partial M'_{i}$,
	\item $\gamma_{i}$ is a simple loop based at $p_{i}$, such that $\gamma_{i}\backslash{\{p_{i}\}}$ lies in the interior of $M_{i}$ and $\{\partial M_{i}\}=2\{\gamma_{i}\}\in \pi_{1}(M_{i})= \langle  [\gamma_{i}]  \rangle $ (likewise for $\gamma'_i$ in $M'_i$, $i\in\{1,..,k\}$).
	
\end{enumerate}

\begin{figure}[!ht]
		\includegraphics[scale=0.33]{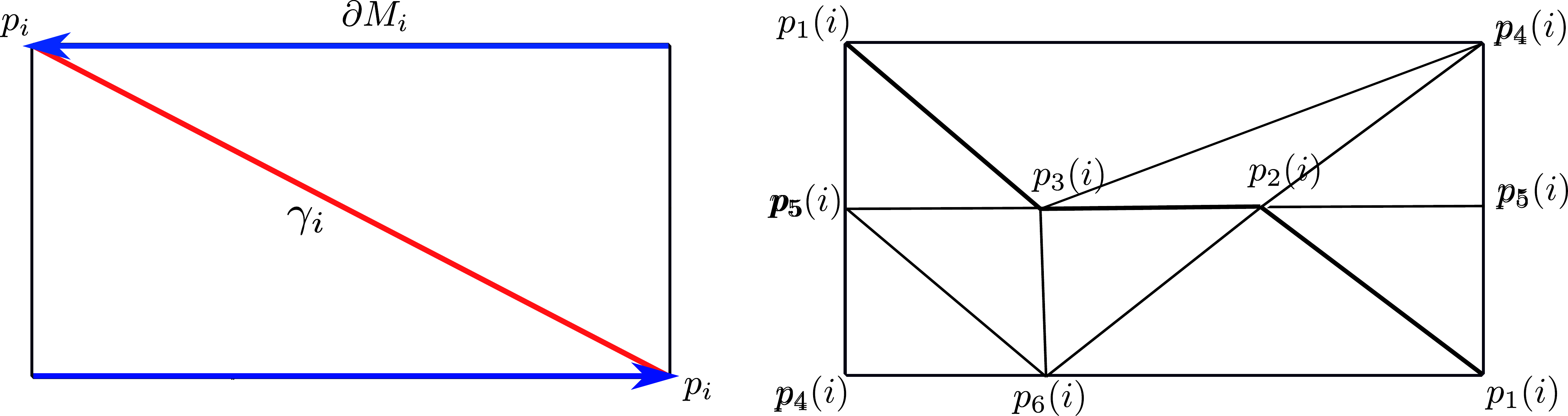}
	
	\caption{\label{bandeTria}The triangulation of the M\"obius band.}
\end{figure}

The triangulation of M\"obius band given in Figure \ref{bandeTria} which is not minimal, is obtained from that of the projective plane, Figure \ref{$RP2$}, from which a $2$-simplex \eg \ $\Delta(p_1,p_4,p_6)$ has been removed. The curve $\gamma_i$ of left-hand side of Figure \ref{bandeTria} corresponds to the concatenation $[p_1(i),p_3(i)]\cup[p_3(i),p_2(i)]\cup[p_2(i),p_1(i)]$ in the triangulation. Also the border $\partial M_i$ corresponds to the concatenation of three adges
$[p_1(i),p_4(i)]\cup[p_4(i),p_6(i)]\cup[p_6(i),p_1(i)]$ in the triangulation. 
The advantage of choosing this triangulation is that both curves, $\gamma_i$ and $\partial M_i$ pass through the vertex $p_i=p_1(i)$ and they have the same simplicial length \ie \ $3$ edges.

\vspace{0.15cm}
One defines two similar telescopic towers $\mathcal{J}_k$ and $\mathcal{L}_k$ as follows,
$$ \left\{
\begin{array}{lll}
\gamma_{0}=\partial\Delta^2_p \\
\mathcal{J}_{1}=M_{0} \\
\mathcal{J}_{k+1}=\mathcal{J}_{k}\underset{\varphi_{k}}{\cup}M_{k}
\end{array}
\right.
$$ 
and
$$ \left\{
\begin{array}{lll}
\gamma'_{0}=\partial\Delta'^2_p \\
\mathcal{L}_{1}=M'_{0} \\
\mathcal{L}_{k+1}=\mathcal{L}_{k}\underset{\psi_{k}}{\cup}M'_{k}
\end{array}
\right.
$$ 
Where, $\varphi_k$ and $\psi_{k}$ are gluing PL-homeomorphisms, satisfying 
\begin{enumerate}
	\item $\varphi_k(\gamma_k)=\partial M_{k-1}$, $\varphi_k(p_k)=p_{k-1}$, we require that all vertices $p_i$ are glued on the same vertex $p$, for all $i\in\{1,...,k\}$,
	\item $\psi_k(\gamma'_k)=\partial M'_{k-1}$, $\psi_k(p'_k)=p'_{k-1}$, the same requirement for the vertices $p'_i$, they are all glued on the same vertex $p$.
\end{enumerate} 

Every gluing homeomorphism will be chosen to be piecewise linear in the sequel. 

\vspace{2mm}
Following the recursive construction the straightforward verification as in \cite{$B^3$} and Lemma \ref{lemma:lem} ensure that both telescopes $\J_k$ and $\L_k $ are simplicial complexes.

\vspace{2mm}
Observe that

\begin{enumerate}
	
	\item[$\bullet$] $\J_1\subset\J_2\subset ...\subset\J_{k-1}\subset\J_k$
	\item[$\bullet$] $\gamma_{0}$ is a deformation retract of $\J_k$ and thus $\pi_{1}(\J_k)=\Z$,
	\item[$\bullet$] $\{\gamma_{i}\}=2^i\{\gamma_{0}\}$ for $i\in\{0,...,k-1\}$.
\end{enumerate}

Also,

\begin{enumerate}
	\item[$\bullet$] $\L_1\subset\L_2\subset ...\subset\L_{k-1}\subset\L_k$
	\item[$\bullet$] $\gamma'_{0}$ is a deformation retract of $\L_k$ and thus $\pi_{1}(\L_k)=\Z$,
	\item[$\bullet$] $\{\gamma'_{i}\}=2^i\{\gamma'_{0}\}$ for $i\in\{0,...,k-1\}$.
	
\end{enumerate}

Let $k$ be the smallest integer such that: $m<2^{k+1}$. The dyadic decomposition of $m$ is written as: $$ m=2^{k_{1}}+2^{k_{2}}+...+2^{k_{s}},$$ where $(k_{j})_{j\in\{0,...,k-1\}}$ are integers, satisfying $0\leq k_{1}<k_{2}<...<k_{s}=k$ and $s\leq k+1$. \\

Let $\xi (m)$ and $\xi'(m)$ be two closed curves, based on $p$ such that $$ \xi(m)=\gamma_{k_{1}}\star\gamma_{k_{2}}\star ...\star \gamma_{k_{s-1}}\star \partial M_{k-1}\in \J_k, $$
and $$ \xi'(m)=\gamma'_{k_{1}}\star\gamma'_{k_{2}}\star ...\star \gamma'_{k_{s-1}}\star \partial M'_{k-1}\in \L_k, $$
we get $ \{\xi(m)\}=m\{\gamma_{0}\} $ and $ \{\xi'(m)\}=m\{\gamma'_{0}\} $. The symbol $\star $ means the concatenation of loops $\gamma_{k_{j}}$ (and $\gamma'_{k_{j}}$) based on $p$, for all $j\in\{1,...s-1\}$. 

\vspace{2mm}
Thus we get two simplicial complexes :  
$$ 
K_w= \left( P'_w\underset{\gamma_{0}}{\cup}\mathcal{J}_k \right)  \underset{\xi(m)\sim \partial D_1^2}{\bigcup} D_1^2
$$ 
and   
$$
K_v= \left( P'_v\underset{\gamma'_{0}}{\cup}\mathcal{L}_k \right)  \underset{\xi'(m)\sim \partial D_2^2}{\bigcup} D_2^2.
$$

\vspace{2mm}

Now fix two $2$-simplexes $T_1\subset D_1 \subset K_w$ and $T_2\subset D_2 \subset K_v$.
They both contain the vertex $p$ (see the dark triangle of Figures \ref{exemple}), and we remove them. Let $K'_w$ and $K'_v$ be two new obtained complexes. 
Note that the both complexes contain the initial bouquet of cercles:
$$
K'_w \hookleftarrow  S \hookrightarrow K'_w,
$$
and let $K' = K'_w \mathop{\cup}\limits_S K'_v$.

Remark that $S \subset K'$ is a deformation retract of $K'$. The curve $\partial T_1$ is homotopic to
the element of $\pi_1(S)$ given by the word $w^m$ in generators $\{a_i\}_{i=1}^n$. In the same way
the curve  $\partial T_2$ is homotopic to
the element of $\pi_1(S)$ given by the word $v^m$ in the same system of generators.

Let $\phi : \partial T_1 \longrightarrow \partial T_2$ be an usual PL-homeomorphisme preserving $p$. The complex
$$
K= K'/ \phi,
$$
which consists in the identification of $\partial T_i$, $i=1, 2$ by $\phi$, is naturally pseudo-simplicial. It is easy to see from construction that
$$
\pi_{1}(K)=\langle a_1, a_2, ..., a_n | w^m=v^m  \rangle .
$$
Remark now that the pseudo-triangulation obtained on $K$ is in fact a triangulation. Let $\{u_i, v_i, p\}$ be the vertexes of $T_i$, $i= 1, 2$. The shortest combinatorial way from $u_1$ to $u_2$ or $v_2$ is $[u_1, p]\cup[p, u_2]$,
respectively $[u_1, p]\cup[p, v_2]$, and it has the length $2$,  idem for $v_1$. 
It is easy to see from construction that any other combinatorial way joining $u_1$ or $v_1$ with $u_2$ or $v_2$ is at least of length $4$. So the gluing by $\phi$ does not degenerate the triangulation.

\vspace{2mm}

Let estimate the number of vertices of the constructed triangulation of $K$. 
 We have 
$$
s_0(K) = s_0(K'_w)+s_0(K'_v)-\left[ s_0(\partial T_1) +2n\right] .
$$

We start with calculation of $s_0(K'_w)$.

\vspace{2mm}
The triangulation of the M\"obius band given by the Figure \ref{bandeTria}, of generator curve $$\gamma_{i}=[p_1(i),p_3(i)]\cup [p_3(i),p_2(i)]\cup [p_2(i),p_1(i)],$$ which is homotopically equivalent to $[p_1(i),p_4(i)]\cup [p_4(i),p_5(i)]\cup [p_5(i),p_1(i)]$, contains $6$ vertices, of which one is the vertex $p$. 

\vspace{2mm}

By construction of $\mathcal{J}_{k}$, the number of vertices that it contains is 
$$
s_{0}(\mathcal{J}_{k})=s_{0}(\gamma_{0})+3k,
$$ where $k$ is the number of M\"obius bands and $s_{0}(\gamma_{0})=3$, then in total we get $s_0(\mathcal{J}_{k})=3k+3$.

\vspace{2mm}

The triangulation of the complex $P'_w$ gives
$$ s_0(P'_w)= 
\left\{
\begin{array}{ll}
1+2t+l+\frac{l}{2} \text{ if } l \text{ is even, } \\
1+2t+l+\frac{l+1}{2} \text{ otherwise. }
\end{array}
\right.
$$
We can write $s_0(P'_w)\leq 1+2t+\frac{3l+1}{2} $ for all $l$, where $1$ corresponds to the vertex $p$, $t\leq n$ is the number of generators that form the word $w$, so $$s_0(P'_w)\leq 1+2n+\frac{3l+1}{2} .$$ 

\vspace{2mm}

The triangulation of the disk $D^2_1$ is different from that of $P_w$, beause its boundary $\partial D^2_1\sim \xi(m)$ is composed of $s$ sections, all different from each other. Each section contains $2$ vertices and the base vertex $p$ so $s_0(\partial D^2_1)=2s+1$, we obtain $ s_{0}(D_1^{2})=3s+1 $.

Figure \ref{trianDisq} illustrates this triangulation for $s=6$.

\begin{figure}[!ht]
	\includegraphics[scale=0.295]{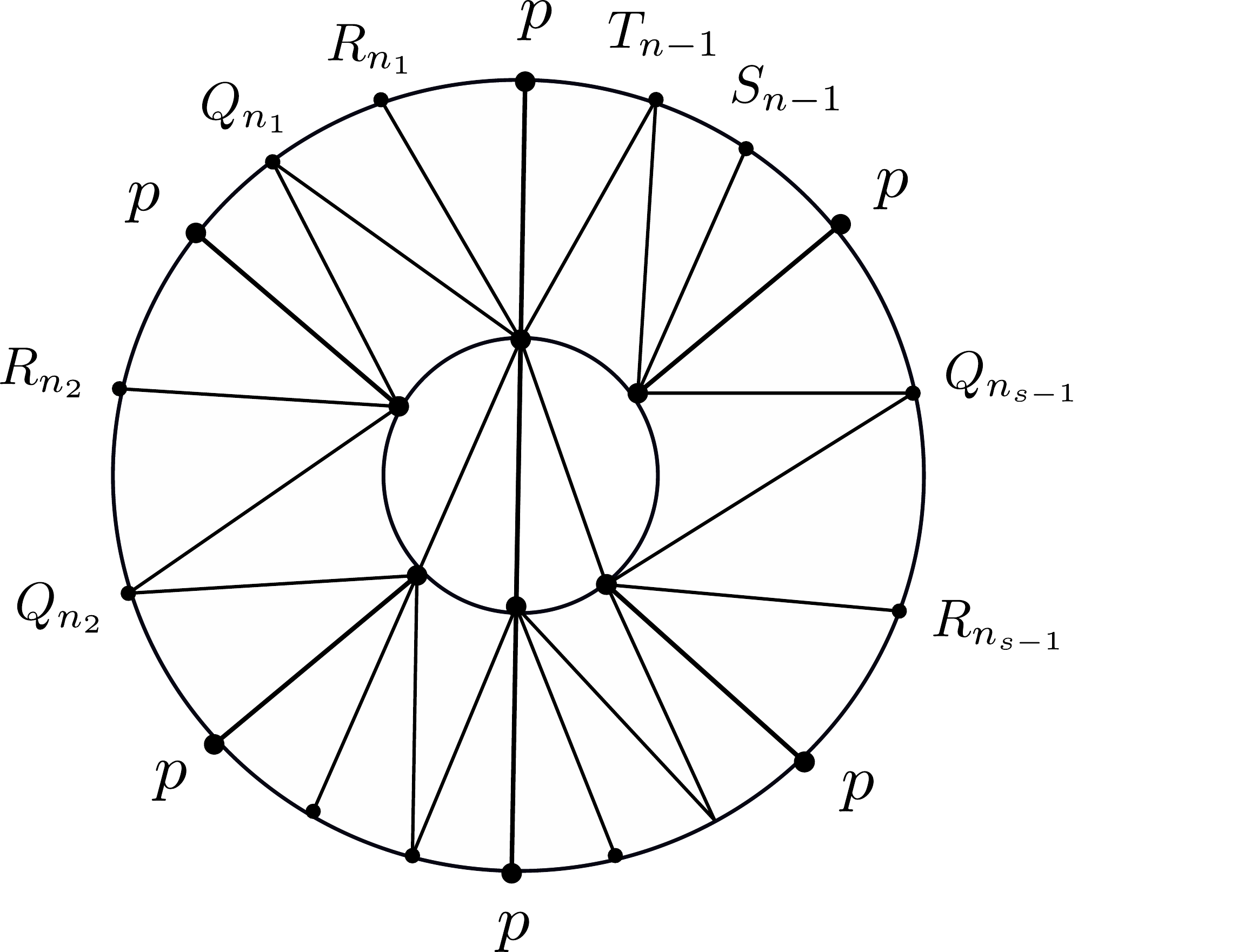}
	\caption{\label{trianDisq}Triangulation of disk $D^{2}_1$ for $s=6$ }
\end{figure} 

Therefore,  
\begin{align*}
s_{0}(K'_w) & =s_0(P'_w)+s_{0}(\mathcal{J}_{k})-s_0(\gamma_{0})+s_{0}(D^{2}_1)-s_{0}(\partial D^{2}_1) \\
& \leq \left( 1+2n+ \frac{3l+1}{2}\right)  +(3k+3)-3+(3s+1)-(2s+1) \\
& \leq 4k+2n+\frac{3l+1}{2} +2 \text{ because } s\leq k+1 
\end{align*}
So $s_0(K'_w)\leq 4\log_2m+2n+\frac{3}{2}l+\frac{5}{2}$

\vspace{2mm}

We repeat the same calculations to estimate $s_0(K'_v)$ and we obtain $$ s_0(K'_v)\leq 4\log_2m+2n+\frac{3}{2}l'+\frac{5}{2} $$

\vspace{2.5mm}

We finally conclude that $s_0(K)$ satisfies
\begin{align*}
s_{0}(K) 
& \leq \left(  4\log_2m+2n+\frac{3}{2}l+\frac{5}{2}\right) +\left(  4\log_2m+2n+\frac{3}{2}l'+\frac{5}{2}\right) -(3+2n) \\
& \leq 8\log_2m+2n+\frac{3}{2}\left( l+l'\right) +2 
\end{align*}

This implies that ${\rm KW}(G)\leq 8\log_2m+2n+\frac{3}{2}\left( l+l'\right) +2  $

\end{proof}

\vspace{1.5mm}

\begin{remark}
If in addition, the word $w$ satisfies $a_{t_i}\neq a_{t_{i+1}}$ for $i\in \{1,...,l-1\}$, then the triangulation of the disk $P_l$ will have fewer vertices than the one given in Figures \ref{exemple} and \ref{exemple2}, as shown in the following Figure for $w=a_1a_2a_1^{-1}a_2$, where $a_1$ and $a_2$ are the generators of the group.  

%\pagebreak

\begin{figure}[!ht]
	\includegraphics[scale=0.365]{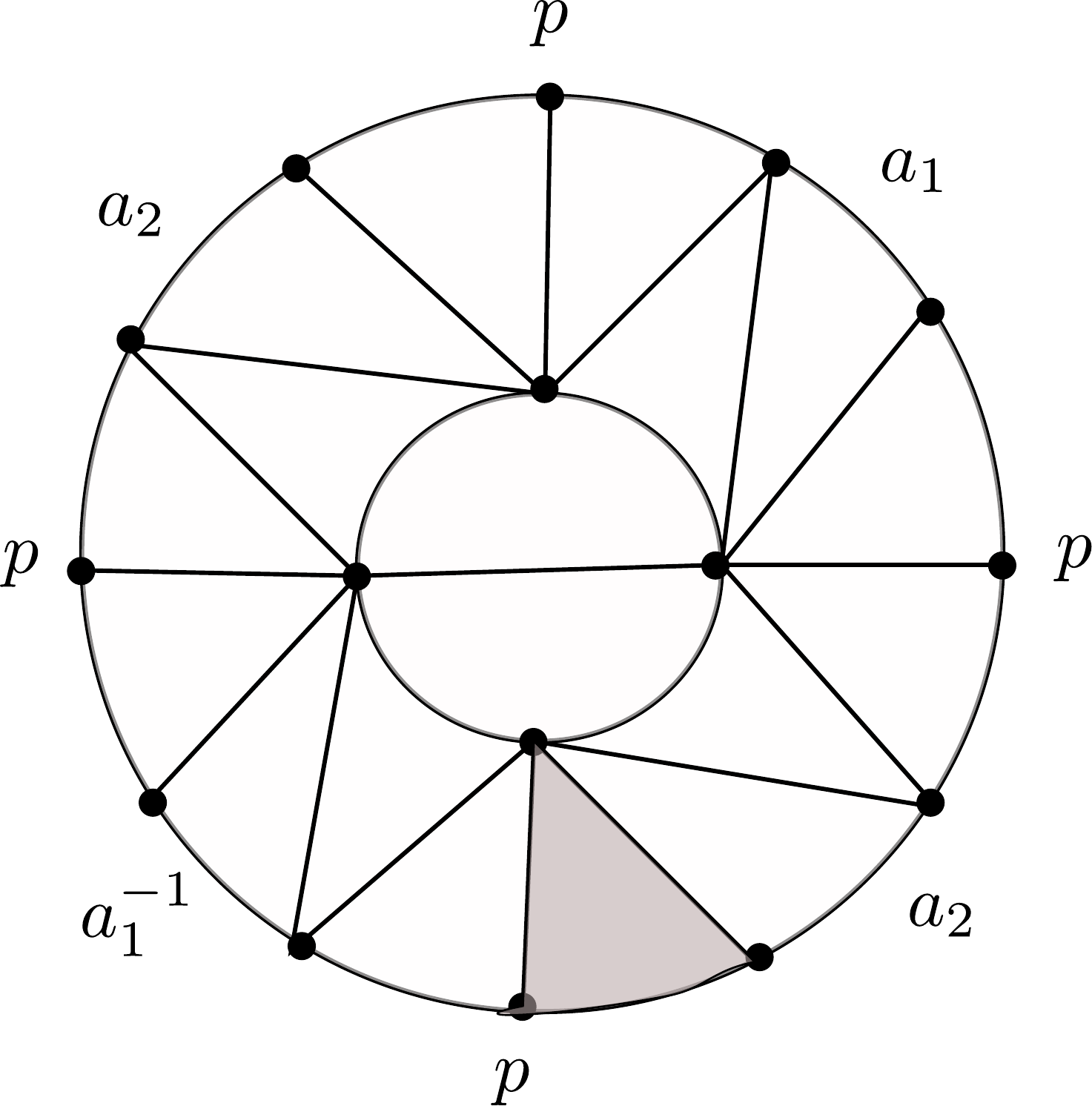}
	\caption{An associated disk $P$ to the word $w=a_1a_2a_1^{-1}a_2$}
\end{figure}
	
\end{remark}

\vspace{1.5mm}

\begin{corollary}\label{Rrelations}
Let $G$ be a finitely presented group defined by $$G=\left\langle a_1,...,a_n \mid w_1^{m_1}=v_1^{m_1},...,w_r^{m_r}=v_r^{m_r}  \right\rangle, $$ where $w_i$ and $v_i$ are two reduced words (they satisfy the properties of Proposition \ref{proposition}) of lengths $l_i$ and $l'_i$, respectively, $m_i$ are non-zero integers for all $i\in\{1,...,r\}$. Then, $${\rm KW}(G)\leq \left[ \overset{r}{\underset{i=1}{\sum}} \left(  8\log_{2}m_i+\frac{3}{2}(l_i+l'_i)\right)\right] +2n+r+1 .$$
\end{corollary}

\begin{proof}[Proof of Corollary \ref*{Rrelations}]
As above, we construct $r$ simplicial $2$-complexes $K_i$, ${i\in\{1,...,r\}}$, such that each of them is associated to the relation $w_i^{m_i}=v_i^{m_i}$. The union of these simplicial $2$-complexes induces a new simplicial $2$-complex $K$ such that $K=\overset{r}{\underset{i=1}{\bigcup}}K_i$ and $\pi_{1}(K)=G$. Therefore, $$ s_0(K) \leq \left[ \overset{r}{\underset{i=1}{\sum}}s_0(K_i)\right] -[(r-1)(2n+1)], $$
where $2n+1$ is the number of vertices which are on $S$. Then
\begin{align*}
s_0(K) & \leq \left[ \overset{r}{\underset{i=1}{\sum}} \left(  8\log_2m_i+2n+\frac{3}{2}\left( l_i+l'_i\right) +2\right) \right] -(2rn-2n+r-1) \\
& \leq \left[ \overset{r}{\underset{i=1}{\sum}} \left( 8\log_2m_i+\frac{3}{2}(l_i+l'_i)\right) \right]+ 2n+r+1,
\end{align*}
and this completes the proof.
\end{proof}

Consider now an Artin group $G_{A}(M)$ of large or extra-large size. Its relations are divided into two natural families :
$$
\E = \{m_{ij} \vert m_{ij}=2k_{ij} < \infty ; i < j \} \ \ 
\text{and} \ \ \O = \{m_{ij} \vert m_{ij}=2k_{ij}+1 < \infty ; i < j \}. 
$$
We detote $r= \vert\E \vert+\vert\O \vert$ the total number of relations and also
$\mu = \mathop{\prod}\limits_{m_{ij}\in \E} m_{ij} \mathop{\prod}\limits_{m_{ij}\in \O} m_{ij}$.

\vspace{1.5mm}

\begin{corollary}\label{artin}
Let $G_{A}(M)$ be an Artin group of large or extra-large size : 
$$
G_{A}(M)=\left\langle  a_{1},...,a_{n}| <a_{i}a_{j}>^{m_{ij}}=<a_{j}a_{i}>^{m_{ij}}, 
i\neq j \text{ et } i,j\in\{1,...,n\}, \right\rangle. 
$$  
Then 
$$ 
{\rm KW}(G_A(M))\leq 8\left({\underset{m_{ij}\in\E}{\sum}}\log_{2}m_{ij} + 
{\underset{m_{ij}\in\O}{\sum}}\log_{2}(m_{ij}-1) \right) +2n-r+1.
$$

Note that this directly implies 
$$
{\rm KW}(G_A(M)) \leq 8\log_{2}\mu +2n-r+1.
$$
\end{corollary} 

\begin{proof}[Proof of Corollary \ref{artin}]
Let $K_{ij}$ be a $2$-simplicial complexe associated to the relations 
$$
<a_{i}a_{j}>^{m_{ij}}=<a_{j}a_{i}>^{m_{ij}}, \ \ m_{ij}\in\E
$$ 
and constructed in the proof of Proposition \ref{proposition}. We consider 
$K_1={\underset{m_{ij}\in \E}{\bigcup}}K_{ij}$
and apply Corollary \ref{Rrelations} where 
$w_{ij}=a_ia_j$, $v_{ij}=a_ja_i$, $|w_{ij}|=l_{ij}=2$ and 
$|v_{ij}|=l'_{ij}=2$ and $k_{ij}=\frac{m_{ij}}{2}$, so
$$ 
s_0(K_1)\leq  \left[ {\underset{m_{ij}\in \E}{\sum}} \left(  8\log_{2}k_{ij}+\frac{3}{2}(l_{ij}+l'_{ij})\right)\right] 
+2n+\vert \E \vert +1
$$ 
Thus we obtain the following upper estimate 

\begin{equation}\label{pair}
s_0(K_1)=8\left({\underset{m_{ij}}{\sum}}\log_{2}m_{ij} \right) +2n-\vert \E \vert+1.
\end{equation}

We proceed now with the relations 
$$
<a_{i}a_{j}>^{m_{ij}}=<a_{j}a_{i}>^{m_{ij}}, \ \ m_{ij}\in\O.
$$ 
For a given $m_{ij}\in\O$ where $m_{ij}=2k_{ij}+1$ and $k_{ij}\in\mathbb{N}\backslash \{0\}$, we have
$$ 
<a_ia_j>^{m_{ij}}=(a_ia_j)^{k_{ij}}a_i \text{ and } <a_ja_i>^{m_{ij}}=(a_ja_i)^{k_{ij}}a_j .
$$ 
Taking $w_{ij}=a_ia_j$, $v_{ij}=a_ja_i$ where $ |w_{ij}|=l_{ij}=2$, $\vert v_{ij}\vert=l'_{ij}=2$
we build, as in the proof of Proposition \ref{proposition}, two complexes
$$
K_{w_{ij}}= \left( P'_{w_{ij}}\underset{\gamma_{0}(m_{ij})}{\bigcup}\mathcal{J}_{k'_{ij}} \right)  \underset{\left( \xi(k_{ij})\vee S^1_{a_i}\right) \sim\partial D_1^2}{\bigcup}  D_1^2,
$$  

$$ 
K_{v_{ij}}= \left( P'_{v_{ij}}\underset{\gamma'_{0}(m_{ij})}{\bigcup}\mathcal{L}_{k'_{ij}} \right)  \underset{\left( \xi'(k_{ij})\vee S^1_{a_j}\right) \sim \partial D_2^2}{\bigcup} D_2^2, 
$$
where $\mathcal{J}_{k'_{ij}}$ and $ \mathcal{L}_{k'_{ij}} $ are two telescopic constructions, both of length $k'_{ij}$, such that $k'_{ij}$ is the smallest integer that satisfies $k_{ij}\leq  2^{k'_{ij}+1}$.
These two complexes are similar to $K_w$ and $K_v$ from the proof of Proposition \ref{proposition} with only
slight modification in the second gluing wich takes into account generators $a_i$ and $a_j$ respectively. 

Like in the proof of proposition \ref{proposition} we deduce by the straightforward calculation 
$$
s_0\left(K_{w_{ij}}\right) \leq s_0\left( P'_{w_{ij}}\right) +s_0\left( \mathcal{J}_{k'_{ij}}\right)
-s_0\left(\gamma_{0}(m_{ij})\right) +s_0\left( D_1^2\right) -s_0\left( \left( \xi(k_{m_{ij}})
\vee S^1_{a_{i}}\right) \sim\partial D_1^2\right).
$$ 
So
\begin{equation}\label{eq:K.1}
s_0\left(K_{w_{ij}}\right) \leq 4\log_{2}k_{ij}+2n+\frac{3}{2}l_{ij}+\frac{7}{2},
\end{equation}
and similarly 
\begin{equation}\label{eq:K.2}
s_0\left(K_{v_{ij}}\right) \leq 4\log_{2}k_{ij}+2n+\frac{3}{2}l'_{ij}+\frac{7}{2},
\end{equation}

and then $
K_{ij}= K_{w_{ij}} \underset{S}\cup K_{v_{ij}}$, where $S=\overset{n}{\underset{s=1}\vee} S^1_{a_s}$.

From upper bounds (\ref{eq:K.1}) and (\ref{eq:K.2}) we obtain 
\begin{align*}
s_0\left(K_{ij}\right)  & \leq s_0\left( K_{w_{ij}}\right) +s_0\left( K_{v_{ij}}\right) -(5+2n) \\
& \leq 8\log_{2}k_{ij}+2n+\frac{3}{2}\left( l_{ij}+l'_{ij}\right) +2 \\ & = 
8\log_{2}(m_{ij}-1)+2n.
\end{align*}

Now, let $K_2 = \underset{m_{ij} \in \O} \bigcup K_{ij}$ be the $2$-simplicial complexe analogous to $K_1$
but corresponding to relations of $\O$. The last upper bound implies 
$$
s_0(K_2)\leq {\underset{m_{ij}\in \O}{\sum}}s_0(K_{ij})-(\vert\O\vert-1)(2n+1),
$$
so
\begin{equation}\label{impair}
s_0(K_2)\leq 8\left({\underset{m_{ij}\in \O}{\sum}}\log_{2}(m_{ij}-1)\right) +2n-\vert\O\vert+1.
\end{equation}
Finally we denote by $K_A$ the simplicial complex obtained by the union of $K_1$ and $K_2$ along the wedge sum $S$. 
It satisfies $\pi_{1}(K_A)=G_A(M)$. The equations (\ref{pair}) and (\ref{impair})  imply 
$$
s_0(K_A)\leq s_0(K_1)+s_0(K_2)-(2n+1),
$$ 
then 
$$ 
s_0(K_A)\leq 8\left({\underset{m_{ij}\in \E}{\sum}}\log_{2}m_{ij} + {\underset{m_{ij}\in\O}{\sum}}\log_{2}(m_{ij}-1) \right) +2n-(\vert\E\vert + \vert\O\vert )+1.  
$$
This complets the proof.
\end{proof}

For the following result, we use the same notations as in Corollary \ref{artin}.

\vspace{1.5mm}

\begin{corollary}\label{coxeterLarge}
Let $G_{C}(M)$ be a Coxeter group of large or extra-large size, as $$
G_{C}(M)=\langle a_{1},...,a_{n}|a_{i}^{2}=e, <a_{i}a_{j}>^{m_{ij}}=<a_{j}a_{i}>^{m_{ji}}, i\neq j \text{ et } i,j\in\{1,...,n\} \rangle,
$$ 
then 

$$
{\rm KW}\left( G_C(M)\right) \leq 8\log_{2}\mu +5n-r+1. 
$$ 
In other words 
$$
{\rm KW}\left( G_C(M)\right) \leq {\rm KW}\left( G_A(M)\right) +3n. 
$$
\end{corollary}

\begin{proof}[Proof of Corollary \ref{coxeterLarge}]
Same as the proof of Corollary \ref{artin}, except that here the $2$-simplicial complex $K_C$, which satisfies $\pi_{1}(K)=G_C(M)$, contains $n$ real projective planes associated to the relations $a_i^2=e$ for all $i\in\{1,...,n\}$. Each of $\mathbb{R}P^2_i$ lies on the circle $S_i^1$ given by   

$$[p_1(i)p_4(i)]\cup[p_4(i)p_5(i)]\cup[p_5(i)p_1(i)],$$  \cf \ Figure \ref{$RP2$}.

Let $K_A$ be the $2$-simplicial complex constructed in the proof of Corollary \ref{artin}, then $$s_0(K_C)=s_0(K_A)+3n,$$ and this concludes the proof.  
\end{proof}

\vspace{1.5mm}

\subsection{The Cyclic and Abelian Groups}

\vspace{1.5mm}

\begin{corollary}\label{cyclique}
	Let $\mathbb{Z}_m=\left\langle a \mid a^m=e \right\rangle $ be a cyclic group of order $m$, so $$ \sqrt[3]{12\log_3m} \leq {\rm KW}(\mathbb{Z}_m)\leq 4\log_2m+4. $$
\end{corollary}

\begin{proof}[Proof of Corollary \ref{cyclique}]
	\
\begin{enumerate}
	\item For the upper bound, we take the first part of the proof of Proposition \ref{proposition}, except that here the word $w=a$ is of length one \ie \ $|w|=1$, so no polygon is associated and the telescopic construction begins directly from the copies of M\"obius bands $\{M_i\}_{i\in\mathbb{N}}$. We get a $2$-simplicial complex $X_m=\mathcal{J}_k\underset{\xi(m)}{\cup}\{D^2/\sim\}$, satisfying $\pi_{1}(K)=\mathbb{Z}_m$, and $$s_{0}(X_m)=s_{0}(\mathcal{J}_k)+s_{0}(D^2/\sim)-s_0(\xi(m)),$$ then ${\rm KW}(\mathbb{Z}_m)\leq 4\log_2m+4$.

	\vspace{2mm}
	\item For the lower bound, we take ${\rm KW}(\mathbb{Z}_{m})=s_0$. The minimal number of $2$-simplices that we can have from $s_0$ vertices, is $\frac{s_0(s_0-1)(s_0-2)}{6}$, and according to Theorem 4.1 of \cite{$B^3$}, we have $$ 2\log_{3}m\leq  \kappa(G) \leq \frac{s_0(s_0-1)(s_0-2)}{6}\leq \frac{s_0^{3}}{6}, $$ then  $$ s_0\geq (12\log_{3}m)^{\frac{1}{3}}, $$
	hence the conclusion.
\end{enumerate}
\end{proof}

\vspace{1.5mm}

\begin{remark}
	Corollary 3.14 implies that $3\leq{\rm KW}(\mathbb{Z}_4)\leq 12$. By using the direct disk attaching construction for $\mathbb{Z}_4$ we obtain ${\rm KW}(\mathbb{Z}_4)\leq 11$.
	
	For the case of $\mathbb{Z}_2\oplus\mathbb{Z}_2$, Corollary 3.8 implies $4\leq {\rm KW}(\mathbb{Z}_2\oplus\mathbb{Z}_2)\leq 13$. We don't know the exact values neither for ${\rm KW}(\mathbb{Z}_4)$ nor for ${\rm KW}(\mathbb{Z}_2\oplus\mathbb{Z}_2)$ as well as we don't know if they coincide. 
\end{remark}

\vspace{1.5mm}

\begin{corollary}
	Let $G$ be a finite abelian group decomposed into direct sum of its invariant factors : $$ G=\mathbb{Z}_{n_{1}}\oplus ...\oplus\mathbb{Z}_{n_{s}}, $$
	such that  $n_i\vert n_{i+1}, \ i=\{1,\dots,s-1\}$ then $$ \left( 12\log_{3} \mid G\mid  \right) ^{\frac{1}{3}}\leq {\rm KW}(G)\leq \left( \frac{3}{s}\log_{2}\mid G\mid +8\right) ^{s}. $$
\end{corollary}

\begin{proof}
	For the lower bound, the reasoning in the same way as in the proof of the corollary \ref{cyclique}. The number of elements of $G$ satisfies $\vert G\vert =n_1n_2\dots n_s$, so $$ {\rm KW}(G)\geq \left( 12\log_{3}|G|\right) ^{\frac{1}{3}}. $$

	On the other hand, we have $$ {\rm KW}(G)\leq \overset{s}{\underset{i=1}{\prod}}{\rm KW}(\mathbb{Z}_{n_{i}})\leq \overset{s}{\underset{i=1}{\prod}}(3\log_{2}n_{i}+8). $$ 
	By using the inequality between the arithmetic and the geometric means we obtain  
\begin{align*}
{\rm KW}(G) & \leq \left( \frac{1}{s} \overset{s}{\underset{i=1}{\sum}} (3\log_{2}n_{i}+8) \right)^{s} \\
& \leq \left( \frac{3}{s} \log_2 \mid G \mid +8 \right)^{s}
\end{align*}

\end{proof}

\vspace{2.5mm}

\section{Complexity and geometric invariants for groups}

\vspace{3mm}

Throughout this part, we shall give the proofs of two main theorems announced in the Introduction, Theorems \ref{th:sigma(G)-ct(G)} and \ref{th:omega(G)}

\vspace{1.5mm}

\subsection{Proof of Theorem \ref{th:sigma(G)-ct(G)}}

\mbox { }

\vspace{2mm}

Consider first inequality, let $k = {\rm KW}(G)$ and $X$ be a $2$-simplicial complex verifying 
$\pi_{1}(X)=G$ and $s_0(X) = k$. Assume that the number of $2$-simplexes in $X$ is minimal possible under considered
conditions. The last means that there is no edge which is adjacent to only one $2$-simplex. Endow $X$ with the special metric $h$ such that: 

\begin{enumerate}
	\item The length of each edge is equal to $1$,
	\item Each $2$-simplex is a spherical equilateral triangle of angle $\frac{\pi}{2}$ and of spherical radius $\frac{2}{\pi}$.
\end{enumerate} 
This means that each edge of triangle is equal to $1$ and its area $S= \frac{2}{\pi}$.

First we prove that $\sys (X, h) \geq 3$. Let $\gamma(t)$ be a systolic geodesic, so it is a simple non-contractible closed curve of length $\sys (X, h)$. 

Let us suppose that  $\gamma(t)$ does not meet vertices of $X$. Because $\gamma(t)$ is a locally minimising geodesic it is a broken spherical line, this means that $\gamma(t)$ intersects any $2$-simplex by some spherical geodesic arc (of radius $\frac{2}{\pi}$).

Let $\Delta=\Delta(p_1, p_2, p_3)$ be a $2$-simplex of vertices $p_1$, $p_2$, $p_3$ meeting $\gamma(t)$. Suppose
$a \in ]p_1, p_2[$ and $b \in ]p_1, p_3[$ are incoming and outcoming points respectively. It means that before 
$\Delta$, $\gamma(t)$ passes through some $2$-simplex $\Delta(p_0, p_1, p_2)$ and after $\Delta$ it passes through some $2$-simplex $\Delta(p_1, p_3, p_4)$. The union of these three simplicies  
$$ 
{\bf U} = \Delta(p_0, p_1, p_2) \cup \Delta(p_1, p_2, p_3) \cup \Delta(p_1,p_3,p_4)
$$ 
forms a part of the hemisphere of radius $\frac{2}{\pi}$ and 
the concatenation of edges  ${\bf u} = [p_0, p_2]\cup [p_2, p_3]\cup [p_3, p_4]$ forms a part of the angle $\frac{3\pi}{2}$ 
lying on the equatorial circle. So the part $\gamma_{\bf U}(t) = \gamma(t) \cap {\bf U}$ is one half of some great cercle
and its length is equal to $2$. Let $\gamma_{\bf U}(t)$ meets the arc  $\bf u$ in two points 
$a' \in ]p_0, p_2[$ and $b' \in ]p_3, p_4[$.

The arc $\gamma_{\bf U}(t)$ is deformable by rotation in $\bf U$ around $a'$ and $b'$ on the new arc 
$\gamma_1 = [a', p_2]\cup [p_2, p_3]\cup [p_3, b'] \subset {\bf U}$. By replacing the arc $\gamma_{\bf U}$ by the arc 
$\gamma_1$ in $\gamma$, we obtain a new closed curve $\tilde{\gamma}$ which is homotopically equivalent to $\gamma$ and of the same length so it is a new systolic geodesic. But this new broken curve is not locally minimal in the neighborhoods of $a'$ and $b'$. Then $\gamma$ is not a systolic geodesic. This contradiction shows that $\gamma$ has to pass through the vertices. 
 
\vspace{1.5mm}

\begin{lemma}\label{lemma:geodesique}
Let $\gamma(t)$ be a geodesic curve in $(X,h)$ which passes through a vertex $q=\gamma(t_0) \in X$. Assume that 
$\gamma(t), t_0<t< t_0+\varepsilon$ does not touch any edge, then $q_2 = \gamma(t_0+2)$ 
is a new vertex and the curve $\gamma(t), t_0\leq t\leq t_0+2$ is deformable $\rm{rel} \{q, q_2\}$ into a concatenation of two edges and this deformation does not change the length of the curves.
\end{lemma}

\begin{proof}[Proof of Lemma \ref{lemma:geodesique}]
	Let $\Delta(q, u, v)$ be a triangle which contains the curve $\gamma(t)$ after it has passed through the vertex $q$. Then $\gamma(t) \cap \Delta(q, u, v)$ is a spherical geodesic arc connecting $q$ and $q_1 \in [u, v]$ and the angle of intersection $\gamma(t) \cap [u, v]$ is $\frac{\pi}{2}$. Passing through $q_1$, $\gamma(t)$ passes in a new triangle $\Delta(u, v, q_2)$. As $\gamma(t)$ is minimal, this implies that $\gamma(t) \cap \Delta(u, v, q_2)$ is a spherical geodesic arc orthogonal to $[u, v]$, then it passes through the vertex $p_2$. This leads to 
	$$
	\gamma_2(t) = \gamma(t) \cap \left( \Delta(q, u, v) \cup \Delta(u, v, q_2)\right), 
	$$
	which is a semi-circle of large radius on the sphere of radius $\frac{2}{\pi}$ whose vertices $q$ and $q_2$ are opposite with respect to the edge $[u,v]$.
	By rotation around $q$ and $q_2$, we can, by homotopy, bring the part $\gamma_2 (t)$ to the concatenation of two edges, for example $[q, u]\cup [u, q_2]$. This deformation does not change neither the homotopy class nor the length. In particular $\text{length}_h(\gamma_2(t)) = 2$. 
 
\end{proof}

By iterating, if necessary, the process of Lemma \ref{lemma:geodesique}, we will see that each closed systolic geodesic can be deformed into a new systolic geodesic that passes only through the edges. It is obvious that a closed non-contractible curve which passes only by edges cannot have a length less than $3$. Then $\sys(X, h) \geq 3$. 

\vspace{2mm}

To complete this first part of the proof, we give the upper bound of the area of $(X, h)$.

We have $\vol(X, h) = s_2(X)\frac{2}{\pi}$, so 
\begin{equation}\label{eq:s_2(X)}
s_2(X) \leq C^3_k \leq \frac{k^3}{6},
\end{equation}
and we get the result.

\vspace{3mm}

We consider now the second inequality of Theorem \ref{th:sigma(G)-ct(G)} giving the lower bound of the systolic area.
Choose $0<\varepsilon< \frac{1}{12}$, by \cite[Theorem 3.5 and Lemme 4.2]{RS08}, there exists an $\varepsilon$-optimal finite Riemannian $2$-complex $(Y, h')$ such that: 
\begin{enumerate}
	\item $\pi_{1}(Y)=G,$
	\item $sys(Y, h')=1,$
	\item $\sigma(Y)<\sigma(G) + \varepsilon,$
	\item For all $\varepsilon< R< \frac{1}{2}$ and for any $y \in Y$, the ball $B(y, R)$ of radius $R$ centered at $y$,
	satisfies 
$$
\vert B(y, R)\vert \geq \frac{1}{4}R^2.
$$ 
\end{enumerate}
Here $\vert B(y, R)\vert$ means the $h$-area of the ball.
Fix $\varepsilon< R < \frac{1}{12}$ and consider a maximal filling of $Y$ by closed disjoint balls of radii $R$ 
and denote by $\{y_i\}_{i=1}^k$ the centers of these balls. By construction, it is obvious that : 
\begin{equation}\label{eq:majoration.ct}
\frac{k}{4}R^2 \leq \vol(Y, h') < \sigma(G) + \varepsilon.
\end{equation}

The balls of radius $2R$ and of the same centers $\{y_i\}_{i=1}^k$ form a cover ${\mathcal U} = \{B(y_i, 2R)\}_{i=1}^k$ of $Y$.

Consider the nerve $\mathcal{N}({\mathcal U})$of this cover and let $N$ be its $2$-skeleton. It is a $2$-simplicial complex with $k$ vertices. If $R< \frac{1}{12}$ then its fundamental group is isomorphic to $G$, see for example \cite[Lemme 3.2]{$B^3$}. From (\ref{eq:majoration.ct}) this implies that 
\[
{\rm KW}(G) \leq k < \frac{4(\sigma(G) + \varepsilon)}{R^2}.
\]
As $\epsilon$ tends to zero and $R$ goes to $\frac{1}{12}$ we deduce ${\rm KW}(G) \leq 576\sigma(G)$ and hence the result.

\subsection{Proof of Theorem \ref{th:omega(G)}}

\mbox { }

\vspace{2mm}

Let ${\rm KW}(G) = k$ and $X$ be an optimal simplicial complex such that $\pi_{1}(X)=G$ and $s_0(X)=k$. 
As $G$ is not free $X$ is obviously $2$-dimensional.

Endow $X$ with a special metric $h$ given in the proof of Theorem \ref{th:sigma(G)-ct(G)}. Recall that each edge of $X$ is of $h$-length equal to $1$, and the $h$-area of each $2$-simplex of $X$ is $S= \frac{2}{\pi}$.

Let $\hat{X}$ be the universal cover of $X$ and $\hat{h}$ is the lift of $h$.

Choose a vertex $q \in \hat{X}$ and consider $\hat{B}_q(R)$ the $\hat{h}$-geodesical ball of radius $R$ and centered at $q$. It is obvious that:

\begin{equation}\label{eq:vol-comparaison}
\vol (\hat{B}_q(R); \hat{h}) \leq s_0(\hat{B}_q(R))\vol (X; h) ,
\end{equation}
where $s_0(\hat{B}_q(R))$ means the number of vertices of $\hat{X}$ in the ball. By applying the Definition of minimal volume entropy (\ref{eq:omega(X)}), we get 
 
\begin{equation}\label{eq:ent-comparaison}
\ent(X, h) \leq \mathop{\lim}_{R \rightarrow \infty}\frac{\log(s_0(\hat{B}_q(R))}{R}.
\end{equation}

\forget

Let's consider a $\hat{h}$-geodesic $\gamma(t)$ in $\hat{X}$ such that $p = \gamma(t_0)$ is a vertex. Suppose that if $t_0 \leq t \leq t_0+\varepsilon$ then $\gamma(t)$ does not coincide with an edge. In this case $\gamma(t)$ passes throught a $2$-simplex $\Delta(p, p_1, p_2) \subset \hat{X}$ and $\gamma(t)$ is this face is a spherical geodesic arc 
On consid\`ere sur $\hat{X}$ une $\hat{g}$-g\'eod\'esique $\gamma(t)$ telle que $p = \gamma(t_0)$ est un sommet. Supposons que $\gamma(t)$ ne co\"{i}ncide pas avec une ar\^ete si $t_0 \leq t \leq t_0+\varepsilon$.
Dans ce cas  $\gamma(t)$ passe dans une $2$-face $\Delta(p, p_1, p_2) \subset \hat{X}$ et la partie de $\gamma(t)$
dans cette face est une arc g\'eod\'esique sph\'erique issue de $p$ qui intersecte l'ar\^ete $[p_1, p_2]$ dans un point
$q$ et les segment g\'eod\'esiques $[p, q]$ et $[p_1, p_2]$  sont perpendiculaires. En passant par $q$ $\gamma(t)$
rentre dans une autre face adjacente \`a $[p_1, p_2]$, l'existence d'une telle face est assur\'ee par la condition 
de minimalit\'e de nombre de $2$-faces dans $X$. Notons $\Delta(p_1, p_2, p')$ cette nouvelle face. La partie de 
$\gamma(t)$ dans cette face est une arc g\'eod\'esique sph\'erique et comme $\gamma(t)$ est localement minimisante
cette arc issue de $q$ est perpendiculaire \`a $[p_1, p_2]$. Alors cette arc passe par le sommet $p'$. Ceci dit que
la partie de $\gamma(t)$ entre $p$ et $p'$ est de longueur $2$ et cette arc est d\'eformable par rotation autour de $p$
et $p'$ dans une arc de m\^eme longueur qui passe par des ar\^ets de $\hat{X}$ : soit la concat\'enation de 
$[p, p_1]$ et $[p_1, p']$ ou bien de $[p, p_2]$ et $[p_2, p']$. 
\forgotten

Let $r$ be a positive integer and $x \in \hat{X}$ be a point such that $\dist_{\hat{h}}(q, x) = r$. Let $\gamma(t), 0 \leq t \leq r$ be a minimal geodesic curve joining $q$ and $x$. The local geometry on $(\hat{X}, \hat{h})$ is the same as that on $(X,h)$. Applying systematically Lemma \ref{lemma:geodesique} to the geodesic curve $\gamma(t), 0 \leq t \leq r$, 
we get that it is deformable into a geodesic curve $\gamma_1(t)$ of the same length which joins $q$ and $x$ and passes through the $1$-skeleton of $\hat{X}$ for at least $0 \leq t \leq r-1$. This implies that the number of vertices in $\hat{B}_q(r)$ is the same that in the ball of radius $r$ in the graph $\rm{Sk}^1(\hat{X})$. Each vertex of this graph is of valence at most $k-1$, so the number of vertices in $\hat{B}_q(r)$ is bounded by $$s_0(\hat{B}_q(r)) \leq \frac{(k-1)^{r+1}}{k-2}.$$ Using (\ref{eq:ent-comparaison}) we deduce that  

\begin{equation}\label{eq:ent-major}
\ent(X, h) \leq \log (k-1). 
\end{equation}
The area of $(X,h)$ is $\frac{2}{\pi} s_2(X)$ and with (\ref{eq:s_2(X)}) we get $\vol (X, h) \leq \frac{k^3}{9}$. This upper bound with (\ref{eq:ent-major}) give us $$ \omega(G) \leq \omega(X, h) = \ent(X, h)(\vol (X, h))^{\frac{1}{2}} \leq \frac{1}{3}\log (k-1)k^{\frac{3}{2}}.$$
This ends the proof.


\begin{thebibliography}{150}


\bibitem{AAK}
K. Adiprasito; S. Avvakumov and R. Karasev, A subexponential size $\R P^n$. Preprint Aug 2021. arXiv:2009. 02703.

\bibitem{AM}
P. Arnoux and A. Marin, The K\"uhnel triangulation of the complex projective plane from the view-point of complex cristallography. II. Mem. Fac. Sci. Kyushu Uni. Ser A 45(2), 167-244 (1991).

\bibitem{$B^3$}
I. Babenko; F. Balacheff and G. Bulteau, Systolic geometry and simplicial complexity for groups.
\emph{J. reine angew. Math.} 757 (2019),  247--277.

\bibitem{BS21}
I. Babenko and S. Sabourau, Minimal volume entropy and fiber growth.
https://arxiv.org/abs/2102.04551, 202

\bibitem{Ba-Lo}
I. B\'ar\'any and L. Lov\'asz, Borsuk's theorem and the number of facets of centrally symmetric polytopes. \emph{Acta Math. Acad. Sci. Hungar}. 40(3-4), (1982), 323-329. 

\bibitem{BM}
E. Borghini and E. G. Minian, The covering type of closed surfaces and minimal triangulations. \emph{J. Combin. Theory Ser}. A 166(2019), 1-10. 

\bibitem{BrieskSai72}
E. Brieskorn and K. Saito, Artin-Gruppen und Coxeter-Gruppen, \emph{Inventiones mathematicae},
17:4 (1972), 245--271.

\bibitem{BC21}
C. Bregman and M. Clay, Minimal volume entropy of free-by-cyclic groups and $2$-dimensional right-angled Artin groups. 
\emph{Math. Ann}. (2021). https://doi.org/10.1007/s00208-021-02211-9.

\bibitem{Char07}
R. Charney, An introduction to right-angled Artin groups. \emph{ Geom. Dedicata} 125 (2007), 141--158.

\bibitem{dug}
J. Dugundji, {\em Topology}, Allyn and Bacon, 1966.

\bibitem{Farber03}
M. Farber, Topological complexity of motion planning, \emph{Discrete Comput.
	Geom.} 29 (2003), 211?221.

\bibitem{GMP}
D. Govc; W. Marzantowicz and T. Pave\v{s}i\'c, \emph{ Estimates of Covering Type and the Number of Vertices of Minimal Triangulations}, Discr. Comp. Geom. 63(2020), 31-48. 

\bibitem{Gro96} 
M. Gromov, Systoles and intersystolic inequalities, in: \emph{ Actes de la table ronde de g\'eom\'trie diff\'erentielle}
(Luminy 1992), S\'emin. Congr. 1, Soci\'et\'e Math\'ematique de France, Paris (1996), 291--362.

\bibitem{hatcher} 
A. Hatcher, \emph{Algebraic topology}. Cambridge University Press, 2002.

\bibitem{KN}
I. Kapovich and T. Nagnibeda, The Patterson-Sullivan embedding and minimal volume entropy
for outer space. \emph{Geometric and Functional Analysis} 17, 4 (2007), 1201--1236.

\bibitem{KW} 
M. Karoubi and Ch. Weibel, On the covering type of a space.
\emph{L'Ens. Math.} 62(2016), 457--474.

\bibitem{KRS} 
M. Katz; Y. Rudyak and S. Sabourau, Systoles of $2$-complexes, Reeb graph, and Grushko decomposition.
\emph{Int. Math. Res. Not.} (2006), Art.~ID~54936, Pages 1-30.

\bibitem{Kur} 
A. G. Kurosh, \emph{The theory of groups}, Chelsea Publishing, New York 1960. 

\bibitem{leray}
J. Leray, Sur la forme des espaces topologiques et sur les points fixes des représentations, \emph{J. Math. Pures Appl}. 24 (1945),95-167 

\bibitem{Lim}
S. Lim, Minimal volume entropy for graphs. \emph{Trans. Amer. Math. Soc.} 360, 10 (2008), 5089--5100. 

\bibitem{LSch}
L. Lusternik and L. Schnirelmann, Sur le probl\`eme de trois g\'eod\'esiques ferm\'ees sur les surfaces de genre 0.
\emph{C.R. Acad. Sci. Paris} 189 (1927), 269-271.

\bibitem{Matveev}
S. Matveev, \emph{Algorithmic Topology of $3$-Manifolds	and Classification}, Springer-Verlag Berlin Heidelberg 2007

\bibitem{McMul}
C. T. McMullen, Entropy and the clique polynomial. \emph{J. Topol.} 8, 1 (2015), 184--212. 

\bibitem{RS08}
Y. Rudyak and S. Sabourau, Systolic invariants of groups and 2-complexes via Grushko decomposition, 
\emph{Ann. Inst. Fourier} 58 (2008), 777--800. 

\bibitem{Pu52} 
P. Pu, Some inequalities in certain non-orientable Riemannian manifolds, \emph{Pacific J. Math.} 2 (1952), 55--71.

\bibitem{ringle}
G. Ringel, \emph{Map Color Theorem}, Springer Grundlehren Ban, 317-326.

\bibitem{Shvarts58} 
A. S. \v{S}varc, The genus of a fiber space. \emph{Dokl. Akad. Nauk SSSR (N.S.)} 119 (1958), 219--222. 

\bibitem{weil}
A. Weil, Sur les Th\'eor\`emes de de Rham, \emph{Commentarii Math. Helv}. 26 (1952), 119-145


\end{thebibliography}
\end{document}